\documentclass{article}

\RequirePackage{amsthm,amsmath,amsfonts,amssymb}
\RequirePackage[numbers]{natbib}
\RequirePackage[colorlinks,citecolor=blue,urlcolor=blue]{hyperref}
\RequirePackage{graphicx}
\RequirePackage{enumitem}
\usepackage{mathrsfs}
\usepackage{graphicx}

\theoremstyle{plain}
\newtheorem{theorem}{Theorem}
\newtheorem{lemma}{Lemma}
\newtheorem{proposition}{Proposition}
\theoremstyle{remark}
\newtheorem{remark}{Remark}
\newtheorem{assumption}{Assumption}
\newtheorem{definition}{Definition}

\title{Stochastic Graphon Games:\\ 
II. The Linear-Quadratic Case}
\author{Alexander Aurell, 
Ren\'e Carmona,
Mathieu Lauri\`ere\\
Department of Operations Research and Financial Engineering,\\
  Princeton University, 
  Princeton, NJ 08544,\\
 \texttt{\{aaurell, rcarmona, lauriere\}@princeton.edu}}

\begin{document}
\maketitle

\begin{abstract}
In this paper, we analyze linear-quadratic stochastic differential games with a continuum of players interacting through graphon aggregates, each state being subject to idiosyncratic Brownian shocks. The major technical issue is the joint measurability of the player state trajectories with respect to samples and player labels, which is required to compute for example costs involving the graphon aggregate. To resolve this issue we set the game in a Fubini extension of a product probability space. We provide conditions under which the graphon aggregates are deterministic and the linear state equation is uniquely solvable for all players in the continuum. The Pontryagin maximum principle yields equilibrium conditions for the graphon game in the form of a forward-backward stochastic differential equation, for which we establish existence and uniqueness. We then study how graphon games approximate games with finitely many players over graphs with random weights. We illustrate some of the results with a numerical example.
\end{abstract}

Keywords (Primary): 91A15, 91A0, 60H10, 60H20               

Keywords (secondary): 60G15, 28E05                 


\section{Introduction} 

Large systems suffer from a growth of complexity with system size that often makes them practically intractable. These challenges arise in control theory, game theory, as well as in many applications in applied mathematics, economics, and engineering. For instance~\cite{geanakoplos2014inflationary,miao2006competitive} consider macroeconomic models with a continuum of infinitesimally small firms or consumers subject to idiosyncratic random shocks. Mean field game (MFG) theory \citep{huang2006large,lasry2006jeux,lasry2006jeux2, carmona2018probabilistic, carmona2018probabilistic2} is a paradigm aimed at the characterization of equilibria in games with a very large number of players. Typically, the game, with a large but finite number of players, is approximated by a limit found by letting the number of players grow to infinity. Much of the MFG theory assumes that players react to one and the same distributional property of the whole population, for example the mean player state. As a resut, it can only approximate games with a high degree of symmetry among the players. 

The theory of graphons provides a framework for the study of very-large systems of agents whose interactions are not necessarily symmetric. See \citep{lovasz2012large} for an expos\'e of the theory. It provides a mathematically rigorous set-up for the analysis of limits of sequences of network games of increasing size. The graphon approximation of a network game is a limit of such a sequence, often referred to as a graphon game. 

Graphon games have recently gained an increasing interest motivated by the study of strategic decision problems on very large networks of non-homogeneous agents. Applications include telecommunications, social networks, electric grids, etc. Static graphon games have been studied in both deterministic and stochastic settings \citep{parise2019graphon, carmona2019stochastic}. 
In the dynamic setting, \citep{delarue2017mean} considers a MFG with an Erd\"os-Renyi graph replacing the standard MFG interaction (a complete graph with uniform weights). The papers \citep{caines2018graphon, caines2019graphon} study decentralized strategies in graphon games. Deterministic linear-quadratic graphon games were studied in \citep{gao2020lqg}. In a setting similar to that of graphon games, a graphon-based model for optimal network interaction was used in \citep{gao2017control, gao2018grapbon, gao2019graphon, gao2019optimal, gao2019spectral} to represent control problems on large networks.
All these works assume that the graph underpinning the network of interactions is dense in the sense that the number of edges (non trivial interactions) is of the same order as the number of vertices (agents). Special cases of sparse networks have been studied, see for example \citep{lacker2019large,feng2020linear,lacker2020case}. Obviously, graphon theory cannot be used in these cases.

In this paper, we study a dynamic, linear-quadratic, stochastic graphon game. The states of a continuum of agents indexed by $I$, a closed bounded interval in $\mathbb{R}$, evolve on the real line while interacting. The agents' state trajectories are given by the following system: for $x\in I$ and $t\in [0,T]$,
\begin{equation}
\label{eq:intro-eq}
    dX^x_t = \Big(a(x)X^x_t + b(x)\alpha^x_t + c(x) \int_I w(x,y)X^y_t \lambda(dy)\Big)dt + dB^x_t,\ X^x_0 = \xi^x,
\end{equation}
where $w(\cdot,\cdot)$ is a bounded graphon, $\lambda$ is a probability measure extending the normalized Lebesgue measure $\lambda_I$ over $I$, $\alpha^x$ is the strategy chosen by player $x$, and $(B^x)_{x\in I}$ are independent standard Brownian motions. The integral $\int_I w(x,y) X^y_t\lambda(dy)$ is the population's influence on player $x$'s state, an aggregate weighted by the graphon. The necessity to integrate with respect to the extending measure $\lambda$ comes from a deep measurability issue connected to the construction of a continuum family of Brownian motions which we would like to be independent of each other. In the macroeconomic literature, it is well known that dealing with a continuum of agents affected by idiosyncratic shocks poses technical challenges related to measurability issues which get in the way of a desirable law of large numbers, see \textit{e.g.},~\cite{judd1985law,bewley1986stationary}. To cope rigorously with these issues, several ways have been investigated. The notion of Fubini extension was introduced to allow for an \emph{Exact Law of Large Numbers} (ELLN) and \citep{sun2006exact} can be viewed as a way to justify the second approach proposed in~\cite{feldman1985expository}. 
At first glance, a natural way to define the aggregate might be $\int_I w(x,y)X^y_t dy$. This is however not well defined in a standard probabilistic setting: in the usual product space carrying a continuum of independent Brownian motions $(B^x)_{x\in I}$ the process $(\omega,x) \mapsto B^x(\omega)$ is not jointly measurable (in the product space of the usual continuum product and the Lebesgue space over the index set), see the references in \citep{sun1998theory}. One way of resolving the measurability issue is to set the model in a Fubini extension of the usual product of the sample probability space and an atomless index probability space $(I,\mathcal{I},\lambda)$ extending the normalized Lebesgue probability space. The theory developed in  \citep{sun2006exact,sun2009individual, podczeck2010existence} grants existence of a probability space $(\Omega\times I, \mathcal{F}\boxtimes \mathcal{I}, \mathbb{P}\boxtimes \lambda)$, called a Fubini extension of the product space, carrying a  collection of essentially pairwise independent (e.p.i.) Brownian motions $(B^x)_{x\in I}$ with sufficient joint measurability (in the extension) for the aggregate to be well defined. By e.p.i., we mean that for $\lambda$-a.e. $x\in I$, $B^x$ is independent of $B^y$ for $\lambda$-a.e. $y\in I$. Moreover, Fubini's theorem for iterated integrals holds in the Fubini extension. Some recent game-theoretical models considering a continuum of agents in a Fubini extension are the rank-based reward models \citep{nutz2018mean, MR4068311, nutz2019mean, yu2020teamwise}, the static graphon game in \citep{carmona2019stochastic}, and the model of \citep{sun2015pure}. 

Our first goal is a rigorous analysis of the system \eqref{eq:intro-eq}. For a strategy profile $(\alpha^x)_{x\in I}$ from an admissible set, essentially defined as the decentralized controls, we show that \eqref{eq:intro-eq} admits a solution well-defined for all $x\in I$ and that the aggregate is a deterministic function of the agent label $x\in I$ and time $t\in[0,T]$. Graphon systems with deterministic aggregates have recently been studied in \citep{bayraktar2020graphon, bayraktar2020stationarity} and are also a feature of the graphon game models cited above. To the best of our knowledge, this paper and \cite{carmona2019stochastic} are the only ones treating stochastic graphon games that do not \textit{a priori} assume the aggregate to be deterministic. As a consequence, our proofs build on the following idea: identify the aggregate as a fixed point argument in an $L^2$-space of random variables on the Fubini extension, and subsequently show that the fixed point must be constant in the sample variable $\omega\in\Omega$. To prove the second step we use the Exact Law of Large Numbers (ELLN) \citep{sun2006exact}. In contrast, the theoretical works that \textit{a priori} assume deterministic aggregates must find them as a fixed point in a space of measure-valued processes.

Next, we solve the linear-quadratic graphon game. In doing so, the deterministic nature of the aggregate turns out to be of paramount importance. In the game, the optimization problem of player $x$ seeking their best response to $\lambda$-a.e. other player following a given strategy profile can be rephrased as a search for their best response to the deterministic aggregate trajectory $\int_I w(x,y)\mathbb{E}[X^y_\cdot]\lambda(dy)$. We derive optimality conditions for the game with a stochastic Pontryagin's type maximum principle. The following forward-backward stochastic differential equation (FBSDE) system, also called the Hamiltonian system, appears in the optimality condition for the equilibrium strategy profile $\underline{\hat{\alpha}} = (\hat{\alpha}^x)_{x\in I}$:
\begin{equation} 
\label{eq:intro2}
\begin{aligned}
&d\hat{X}^x_t 
= 
\partial_p H^x(t,\hat{X}^x_t, \hat{\alpha}^x_t, p_t)dt + dB^x_t,& &\hat{X}_0^x = \xi^x,
\\
&dp^x_t 
= 
-\partial_\chi H^x(t,\hat{X}^x_t, \hat{\alpha}^x_t, p_t)dt + q^x_tdB^x_t,& &p^x_T = \partial_\chi h^x(\hat{X}_T^x, \hat{Z}^x_T),\quad t\in [0,T], x\in I.
\end{aligned}
\end{equation}
We show that it is well-defined: for any finite time horizon $T$ there is a unique solution that solves the Hamiltonian system for all $x\in I$. In \eqref{eq:intro2} $H^x$ and $h^x$ are the Hamiltonian and terminal cost of player $x$, respectively. With a solution of \eqref{eq:intro2} at hand, the players can construct an admissible strategy profile of decentralized controls which is a Nash equilibrium to the graphon game. Under assumptions on the convexity of $H^x$ and $h^x$, $x\in I$, the Nash equilibrium is unique.

Many applications of interest involve models with finitely many players. Nevertheless, the search for Nash equilibria is most often prohibitive for all practical purposes, justifying the quest for approximation results similar to those in the MFG theory. We find a class of $N$-player games over fully connected interaction networks with random weights for which the graphon game provides an approximation and the graphon game equilibrium is the limit behavior. In these finite player games, the aggregate for player $k$ is the sum $\frac{1}{N}\sum_{\ell=1}^N w(i_k,i_\ell)X^{\ell,N}_t$ where $i_1,i_2,\dots$ are randomly sampled from $[0,1]$ (according to some distribution). In general, the search for approximate equilibria for such $N$-player games is outside the scope of classical MFG theory (except for the piecewise constant graphon). More specifically, we show that a graphon game Nash equilibrium is an approximate Nash equilibrium for a $N$-player game as described above. In the other direction, Nash equilibria of a sequence of such finite player games converges to a graphon game equilibrium as $N\rightarrow \infty$. Under a continuity assumption on the graphon, we use the the Law of Iterated Logarithms for Banach space valued random variables to show that the rate of convergence and  the approximation error are both of order almost $1/\sqrt{N}$ with probability $1$ with the Law of Iterated Logarithms for Banach space valued random variables.

The paper is structured as follows. In section~\ref{sec:prels} we introduce the notation and the necessary background on graphons and Fubini extensions. 
We analyze equations \eqref{eq:intro-eq} and \eqref{eq:intro2} in Section~\ref{sec:graphon-games}. The  convergence results can be found in Section~\ref{sec:convergence}. Section~\ref{sec:examples} treats a special case of the game for which the solution can be computed semi-explicitly, and presents some solved examples. Most proofs are deferred to the appendix.

\section{Preliminaries}
\label{sec:prels}

\subsection{The graphon}
Let $I\subset \mathbb{R}$ be a closed bounded interval.  The set $I$ is an index set labeling a continuum of agents. The Lebesgue space over $I$ is denoted by $(I,\mathcal{B}_I, \lambda_I)$ and $L^2(I)$ denotes the space of $\lambda_I$-equivalence classes of functions $X : I \rightarrow \mathbb{R}$ such that  $\int_I|X(x)|^2\lambda_I(dx) < \infty$. 
Hereafter, for a function $X : I \rightarrow \mathbb{R}$, we will often use the notation $X^x = X(x)$ for $x \in I$, and sometimes write $X^\cdot = X(\cdot)$ to stress the fact that $X$ is a function of the index. When the context is clear the dot is omitted.

A graphon is a symmetric, measurable function $w: I\times I \rightarrow [0,1]$. It defines an integral operator $W: L^2(I) \rightarrow L^2(I)$:
$$
[Wf](x) = \langle w(x,\cdot), f\rangle_{\lambda_I} := \int_I w(x,y)f(y)\lambda_I(dy),\quad x\in I,\ f\in L^2(I).
$$
The operator $W$ is a symmetric Hilbert-Schmidt operator: there exists an orthonormal basis in $L^2(I)$ of eigenfunctions $\{\varphi_i\}_{i=1}^\infty$ of $W$ such that the eigenvalues $\lambda_k$ are real and square summable and:
\begin{equation}
\label{eq:graphon_decomp}
    \left[Wf\right](x) 
    = \sum_{k=1}^\infty \lambda_k\varphi_k(x)\langle f, \varphi_k \rangle_{\lambda_I},\quad f\in L^2(I), \;\lambda_I\text{-a.e.}\;\; x\in I,
\end{equation}
and the Hilbert-Schmidt norm of $W$ is given by $\|W\|_2 = \|w\|_{L^2(I\times I)}=[\sum_{k\ge 1}\lambda_k^2]^{1/2}$. A standard reference for graphon theory is Lov\'asz's book \citep{lovasz2012large}.

\subsection{Fubini extensions}
\label{sec:fub}
Much of the analysis in this paper relies on the existence of an uncountable collection of essentially pairwise independent (e.p.i.) random variables jointly measurable in the sample $\omega\in\Omega$ and the label  $x\in I$. The theory of Fubini extensions was developed to facilitate this existence. We recite its formal definition here. 
\begin{definition}
\label{def:fub}
If $(\Omega, \mathcal{F}, \mathbb{P})$ and $(I, \mathcal{I}, \lambda)$ are probability spaces, a probability space $(\Omega\times I, \mathcal{W}, \mathbb{Q})$ extending the usual product space $(\Omega\times I, \mathcal{F}\otimes \mathcal{I}, \mathbb{P}\otimes \lambda)$ is said to be a Fubini extension if for any real-valued $\mathbb{Q}$-integrable function $f$ on $(\Omega\times I, \mathcal{W})$
\begin{enumerate}[label={(\roman*)}]
    \item\label{def:Fubini-i} the two functions $f_x : \omega \mapsto f(\omega,x)$ and $f_\omega : x \mapsto f(\omega,x)$ are integrable, respectively, on $(\Omega, \mathcal{F},\mathbb{P})$ for $\lambda$-a.e. $x\in I$, and on $(I,\mathcal{I},\lambda)$ for $\mathbb{P}$-a.e. $\omega\in \Omega$;
    \item\label{def:Fubini-ii} $\int_\Omega f_x(\omega) d\mathbb{P}$ and $\int_If_\omega(x)d\lambda(x)$ are integrable, respectively, on $(I,\mathcal{I},\lambda)$ and $(\Omega, \mathcal{F},\mathbb{P})$, with $\int_{\Omega\times I}f(\omega,x)d\mathbb{Q}(\omega,x) = \int_I\left(\int_\Omega f_x(\omega)d\mathbb{P}(\omega)\right)d\lambda(x) = \int_\Omega\left(\int_If_\omega(x)d\lambda(x)\right)d\mathbb{P}(\omega)$.
\end{enumerate}
\end{definition}

\begin{remark}
\label{re:Bochner}
Given a Fubini extension as above, if $E$ is a separable Banach space and $f$ is a strongly measurable $E$-valued function on $(\Omega\times I, \mathcal{W})$, then properties \ref{def:Fubini-i} and \ref{def:Fubini-ii} of Definition \ref{def:fub} still hold as long as we interpret measurability as strong measurability and integrability in the sense of Bochner integrals. This claim is an immediate consequence of the fact that an $E$-valued function $\varphi$ on a measure space $(M,\mathcal{M})$ is strongly measurable if and only if for each element $e^*$ of the dual $E^*$ of $E$, the real valued function $M\ni m\mapsto \langle e^*,\varphi(m)\rangle \in\mathbb{R}$ is measurable, and if $\mu$ is a measure on $(M,\mathcal{M})$, the Bochner integral $\int_M\varphi(m)\mu(dm)$ is characterized, whenever it exists as an element of $E$, by
$$
\langle e^*,\int_M\varphi(m)\mu(dm)\rangle=\int_M\langle e^*,\varphi(m) \rangle\mu(dm),\qquad e^*\in E^*.
$$
\end{remark}

In the next theorem we recite some results on the existence of a Fubini extension carrying a continuum of e.p.i. jointly measurable random variables. It introduces a particular sample space, $(\Omega,\mathcal{F}, \mathbb{P})$, an index space $(I, \mathcal{I}, \lambda)$, and a Fubini extension of their product.
See \citep[Prop. 5.6]{sun2006exact} for the construction of the sample space.
The index space is the extension of $(I,\mathcal{B}_I, \lambda_I)$ explicitly constructed in \citep{sun2009individual}, where we note that $\mathcal{I}$ is non-countably generated. 
The following existence result is taken from \citep[Thm. 1]{sun2009individual}.

\begin{theorem}
\label{thm:existence}
Let $I$ be the unit interval and $S$ a Polish space. There exists a probability space $(I,\mathcal{I}, \mathcal{\lambda})$ extending $(I,\mathcal{B}_I,\lambda_I)$, a probability space $(\Omega,\mathcal{F},\mathbb{P})$, and a Fubini extension $(\Omega\times I, \mathcal{F}\boxtimes \mathcal{I},\mathbb{P}\boxtimes \lambda)$ of $(\Omega\times I, \mathcal{F}\otimes \mathcal{I},\mathbb{P}\otimes \lambda)$ such that for any measurable mapping $\varphi$ from $(I,\mathcal{I},\lambda)$ to $\mathcal{P}(S)$, the set of Borel probability measures on $S$, there is an $\mathcal{F}\boxtimes\mathcal{I}$-measurable process $f : \Omega\times I \rightarrow S$ such that the random variables $f_x = f(\cdot,x)$ are e.p.i. and $\mathbb{P}\circ f^{-1}_x = \varphi(x)$ for
$x\in I$.
\end{theorem}

Let $T>0$ be a finite time horizon and let $E := C([0,T]; \mathbb{R})$ be the space of real-valued continuous functions on $[0,T]$ equipped with the topology of the uniform convergence. Let $S$ be the Polish space $S := E \times \mathbb{R}$, the Borel $\sigma$-field over $S$ is $\mathcal{B}(S) = \mathcal{B}(E)\otimes \mathcal{B}(\mathbb{R})$. We use the notation $\sigma=([\sigma]_1,[\sigma]_2)$ to emphasize the two components of $\sigma\in S$. For $x\in I$, let $\varphi(x)=\mu(x)\otimes\nu(x)$ where $\mu(x)$ is the Wiener measure on $E$ for each $x\in I$ and 
$\nu : I \to \mathcal{P}(\mathbb{R})$ is an $\mathcal{I}$-measurable function. The probability measure $\nu(x)$ models the initial probability distribution of player $x$'s state. By Theorem~\hyperref[thm:existence]{\ref{thm:existence}}, there exists a $\mathcal{F}\boxtimes \mathcal{I}$-measurable process $\mathbb{B} : \Omega\times I \rightarrow S$ with random variables $\mathbb{B}^x(\cdot) = (B^x(\cdot),\xi^x(\cdot))$ e.p.i. and such that 
the law of $(B^x,\xi^x)$ is $\mu(x) \otimes \nu(x)$ for all $x\in I$. 
We denote by $\mathbb{F}^x$ the filtration generated by $\mathbb{B}^x$.

We write $L^2_{\boxtimes}(\Omega\times I; E)$ for the space of equivalence classes of $(\mathcal{F}\boxtimes\mathcal{I}, \mathcal{B}(E))$-measurable functions which are $\mathbb{P}\boxtimes\lambda$-square integrable 
, \textit{i.e.}, $\varphi \in L^2_\boxtimes(\Omega\times I; E)$ if
$$
    \int_{\Omega \times I} \|\varphi^x(\omega)\|^2_E \mathbb{P} \boxtimes \lambda(d\omega, dx) < +\infty,
$$
and $L^2_\boxtimes(\Omega\times I)$ when $E=\mathbb{R}$.
We write $L^2_\lambda(I; E)$ and $L^2(I; E)$ for the space of equivalence classes of $(\mathcal{I}, \mathcal{B}(E))$ and $(\mathcal{B}(I), \mathcal{B}(E))$ measurable functions which are $\lambda$ and $\lambda_I$-square integrable, respectively. For any $\varphi \in L^2_{\boxtimes}(\Omega\times I)$, we use the notation:
$$
    \mathbb{E}^{\boxtimes}\left[\varphi\right] := \int_{\Omega\times I}\varphi^x(\omega) \mathbb{P}\boxtimes \lambda(d\omega,dx).
$$
Expectation without a superscript, $\mathbb{E}$, refers to the integral taken with respect to $\mathbb{P}$.

We extend the domain of the graphon operator to $L^2_\boxtimes(\Omega\times I)$:
\begin{equation}
\label{eq:w_on_X}
\begin{aligned}
    &W : L^2_\boxtimes(\Omega\times I) \ni X \mapsto [WX] : \Omega\times I \ni (\omega,x) \mapsto \int_I w(x,y)X^y(\omega)\lambda(dy) \in \mathbb{R}.
\end{aligned}
\end{equation}
Naturally, we inquire if $([WX_t])_{t\in[0,T]}$ is measurable when $X\in L^2_\boxtimes(\Omega\times I; E)$ if we understand $X_t$ as $\Omega\times I\ni (\omega,x)\mapsto X_t(\omega,x)=X^x(\omega)_t\in\mathbb{R}$.

\begin{lemma}
\label{lemma:extended-W}
If $X\in L^2_\boxtimes(\Omega\times I; E)$, 
\begin{enumerate}[label=(\roman*)]
    \item\label{lemma:extended-W-i} for $x\in I$ and $\mathbb{P}$-almost every $\omega\in\Omega$, the Bochner integral  $\int_I w(x,y)X^y(\omega)\lambda(dy)$ defines an element $[WX](x,\omega)$ of $E$, 
    \item\label{lemma:extended-W-ii} the mapping $I\times\Omega\ni(x,\omega)\mapsto [WX](x,\omega)$
is measurable with respect to the completion for $\lambda_I\otimes\mathbb{P}$ of the product $\sigma$-field $\mathcal{B}_I\otimes\mathcal{F}$, 
\item\label{lemma:extended-W-iii} $[WX]$ so defined provides an extension of the graphon operator $W$ to a bounded operator on $L^2_\boxtimes(\Omega\times I; E)$ of norm at most $1$.
\end{enumerate}
\end{lemma}
Notice that Remark \ref{re:Bochner} implies that $[WX](x,\omega)_t=[WX_t](x,\omega)$ as defined by formula \eqref{eq:w_on_X}.
\begin{proof}
If $x\in I$ is fixed, $\|w(x,y)X^y(\omega)\|_E\le \|X^y(\omega)\|_E$ which is $\mathbb{P}\boxtimes \lambda$-integrable. Moreover, $x$ being fixed, as a function of $(\omega,y)$, $w(x,y)X^y(\omega)$
is $\mathcal{F}\boxtimes\mathcal{I}$-measurable so that the variation of the definition of Fubini extension given in Remark \ref{re:Bochner} implies that the Bochner integral  $\int_I w(x,y)X^y(\omega)\lambda(dy)$ defines an element $[WX](x,\omega)$ of $E$ which happens to be a $\mathcal{F}$-measurable function of $\omega\in\Omega$.

Using a sequence of functions of the form $(x,y)\mapsto\sum_{1\le i\le k}\varphi_i(x)\psi_i(y)$
where $k$ is an integer and $\varphi_i$ and $\psi_i$ are $\mathcal{B}_I$-measurable real valued functions to approximate $w(x,y)$ almost everywhere on $I\times I$ for the Lebesgue measure, and functions of the form $\sum_{1\le i\le k}{\bf 1}_{U_i}(\omega,y)e_i$ for some $U_i\in\mathcal{F}\boxtimes\mathcal{I}$ and $e_i\in E$, to approximate $X^y(\omega)$  $\mathbb{P}\boxtimes\lambda$-a.s., we see that $[WX](x,\omega)$ is in fact almost surely equal to a $\mathcal{B}_I\otimes\mathcal{F}$-measurable function as claimed in \ref{lemma:extended-W-ii}.

As for \ref{lemma:extended-W-iii}, it follows from the following inequalities:
\begin{equation*}
    \begin{split}
        \int_{\Omega\times I}\|[WX](x,\omega)\|^2_E\;\mathbb{P}\boxtimes \lambda(d\omega,dx)
        &=\int_{\Omega\times I}\Bigl\|\int_I w(x,y)X^y(\omega)\lambda(dy)\Bigr\|^2_E\;\mathbb{P}\boxtimes \lambda(d\omega,dx)\\
        &\le \int_{\Omega\times I}\Bigl(\int_I \|X^y(\omega)\|^2_E\lambda(dy)\Bigr)\;\mathbb{P}\boxtimes \lambda(d\omega,dx)\\
        &= \int_{\Omega\times I} \|X^y(\omega)\|^2_E\mathbb{P}\boxtimes \lambda(d\omega,dy)
    \end{split}
\end{equation*}
where we used once more the defining property of a Fubini extension.
\end{proof}

We conclude this section with a note about the Brownian motion in the Fubini extension. As an $E$-valued Gaussian vector, $B$ has exponential moments, so $B \in L^2_{\boxtimes}(\Omega\times I; E)$. Furthermore for each $t\in [0,T]$, the (weak) Exact Law of Large Numbers (ELLN) \citep[Corollary 2.10]{sun2006exact} applies to $B_t$ since $(B_t^x)_{x\in I}$ are e.p.i. and $B_t$ is $\mathbb{P}\boxtimes\lambda$-integrable. By the ELLN, for each $A\in\mathcal{I}$ and $t\in[0,T]$ we have:
\begin{equation*}
\int_A B^x_t (\omega)\lambda(dx) = \int_A \mathbb{E}[B_t^x]\lambda(dx) = 0,\quad \mathbb{P}\text{-a.s}.
\end{equation*}

\section{Linear-quadratic stochastic graphon games}
\label{sec:graphon-games}

 In this section, we introduce a game with a continuum of players indexed by $I$. Each player's goal is to choose the best  admissible strategy. Below we define the notion of admissibility, formalize the players' state dynamics and costs/rewards. We postpone the specification of what is meant by "best strategy", and its computation, to Sections \ref{sec:pmp} and \ref{sec:FBSDE}.

\subsection{Linear graphon dynamics}

A family $(\alpha^x)_{x\in I}$ of $\mathbb{F}^x$-progressively measurable real-valued processes is called a strategy profile.
\begin{definition}
Let $\underline{\mathcal{A}}$ be the set of real-valued, $\mathcal{I}\otimes \mathcal{B}([0,T]) \otimes \mathcal{B}(S)$-measurable functions $\underline\alpha$ on $I\times [0,T]\times S$ satisfying
\begin{itemize}
    \item[(progressive)] for every $(x,t) \in I\times [0,T]$, $\underline\alpha(x,t,\sigma) = \underline\alpha(x,t,\sigma')$ if $[\sigma]_1(s) = [\sigma']_1(s)$, $0\leq s\leq t$ and $[\sigma]_2 = [\sigma']_2$;
    \item[(integrability)] $\mathbb{E}^\boxtimes\left[\int_0^T|\underline{\alpha}(\cdot,t, \mathbb{B}^\cdot)|^2dt\right]<\infty$ and, for all $x\in I$, $\mathbb{E}[\int_0^T |\underline\alpha(x,t,\mathbb{B}^x)|^2dt]<\infty$. 
\end{itemize}
\end{definition}
The strategy profile $(\alpha^x)_{x\in I}$ is called admissible if there is an $\underline\alpha \in \underline{\mathcal{A}}$ such that $\alpha^x_\cdot = \underline\alpha(x, \cdot, \mathbb{B}^x)$ for $x\in I$.
With some abuse of notation, we will write $(\alpha^x)_{x\in I} = \underline\alpha \in \underline{\mathcal{A}}$ when $(\alpha^x)_{x\in I}$ is admissible. 

For each $x\in I$, we define $\mathcal{A}(x)$ as the set of $\mathbb{F}^x$-progressively measurable square-integrable processes $(\alpha_t)_{0\le t\le T}$, \textit{i.e.}, satisfying $\mathbb{E}\int_0^T|\alpha_t|^2dt < \infty$. When the player population plays an admissible strategy profile, the strategy of player $x\in I$ is an element of $\mathcal{A}(x)$. Furthermore, player $x$ can switch their strategy to any $\beta\in \mathcal{A}(x)$ without affecting the overall $\mathcal{I}\otimes\mathcal{B}([0,T])\otimes\mathcal{B}(S)$-measurability of the strategy profile.

The next assumption is in force throughout the rest of the paper.

\begin{assumption}
\label{assump:sec3.1}
\begin{enumerate}[label=(\roman*)]
    \item 
    \label{assump:sec3.1-i}
    $\nu(x)$ has a finite second moment uniformly bounded in $x$
    \item 
    \label{assump:sec3.1-ii}
    The coefficient functions $a,b,c : I \rightarrow \mathbb{R}$ are $\mathcal{I}$-measurable and bounded
\end{enumerate} 
\end{assumption}

The rest of this section is devoted to the existence and uniqueness of a solution to the graphon SDE system
\begin{equation}
\label{eq:linear_dynamics_intro}
    dX_t^{\underline{\alpha},x} = \left(a(x)X_t^{\underline{\alpha},x}+b(x)\alpha^x_t + c(x)Z^{\underline{\alpha},x}_t\right)dt + dB^x_t,\ t\in[0,T], \quad X_0^{\underline{\alpha},x} = \xi^x,
\end{equation}
 defined for all $x\in I$, where $\alpha^x_\cdot = \underline{\alpha}(x, \cdot, \mathbb{B}^x)$ for a given $\underline{\alpha}\in \underline{\mathcal{A}}$. $Z^{\underline{\alpha},x}$ is the graphon-weighted aggregate of the continuum of player states, defined in line with aggregates in economic theory, and for now only formally, as:
\begin{equation}
\label{eq:definition_of_Z}
Z^{\underline{\alpha},x}_t := \int_I w(x,y)X^{\underline{\alpha},y}_t\lambda(dy) = [WX^{\underline{\alpha},\cdot}_t](x),\quad t\in[0,T].
\end{equation}

\begin{lemma}
\label{lemma:well-posedness-of-X}
If $z\in L^2_{\boxtimes}(\Omega\times I; E)$ and $\underline{\alpha}\in\underline{\mathcal{A}}$, the stochastic integral equation
\begin{equation}
\label{eq:lemma2-state-eq}
\mathbb{X}_t^{\underline\alpha,z,\cdot} = \mathbb{X}^{\underline\alpha,z,\cdot}_0 + \int_0^t\left(a(\cdot)\mathbb{X}_s^{\underline\alpha,z,\cdot} + b(\cdot)\alpha^\cdot_s + c(\cdot)z_s(\cdot)\right)ds + B^\cdot_t,\ t\in[0,T], \quad \mathbb{X}^{\underline\alpha,z,\cdot}_0 = \xi^{\cdot}
\end{equation}
has a unique solution $\mathbb{X}^{\underline\alpha,z}_\cdot=(\mathbb{X}_\cdot^{\underline\alpha,z,x}(\omega); (\omega,x)\in \Omega\times I)$ in $L^2_\boxtimes(\Omega\times I; E)$.
\end{lemma}

A solution to \eqref{eq:lemma2-state-eq} is any $E$-valued process on $\Omega\times I$ that satisfies the equation for $\mathbb{P}\boxtimes\lambda$-a.e $(\omega,x)\in \Omega\times I$. It is unique if $\mathbb{P}\boxtimes\lambda(\mathbb{X}_t = \widetilde{\mathbb{X}}_t,\ t\in[0,T]) = 1$ whenever $\mathbb{X}$ and $\widetilde{\mathbb{X}}$ are solutions. The proof of Lemma~\ref{lemma:well-posedness-of-X} is an adaptation of the standard existence and uniqueness proof for SDEs and is omitted. Next we characterize the aggregate as the unique fixed point to the following map:
\begin{equation}
\label{eq:U}
\begin{aligned}
    &U^{\underline\alpha} : L^2_{\boxtimes}(\Omega\times I; E) 
    \rightarrow 
    L^2_{\boxtimes}(\Omega\times I; E),
    \\
    &z
    \mapsto 
    U^{\underline\alpha}z : (\omega,x) \mapsto
    \Big(\int_I w(x,y)\mathbb{X}_t^{\underline\alpha,y,z}(\omega)\lambda(dy)\Big)_{t\in[0,T]}
\end{aligned}
\end{equation}
where $\underline{\alpha}\in \mathcal{A}$ is fixed and $\mathbb{X}^{\underline\alpha,z}$ is the unique solution to \eqref{eq:lemma2-state-eq} as given by Lemma~\ref{lemma:well-posedness-of-X}. 

\begin{proposition}
\label{prop:fp-U}
For each admissible strategy profile  $\underline{\alpha}\in\underline{\mathcal{A}}$, the mapping $U^{\underline\alpha}$ is well-defined and has a unique fixed point which we shall denote $\mathbb{Z}^{\underline\alpha}$.
\end{proposition}

A proof of Proposition~\ref{prop:fp-U} is found in the appendix. Thus, there exists a unique solution in $L^2_\boxtimes(\Omega\times I; E)$ to the graphon SDE system~\eqref{eq:linear_dynamics_intro}--\eqref{eq:definition_of_Z} with $\underline{\alpha}\in \mathcal{A}$ fixed and the aggregate defined as the fixed point of $U^{\underline{\alpha}}$. The next result further specifies the structure of the solution. In summary, the aggregate must be almost surely deterministic and there is a unique version of the solution that solves the linear graphon SDE \eqref{eq:linear_dynamics_intro} for every $x\in I$ and whose corresponding aggregate is deterministic everywhere.

\begin{theorem}
\label{thm:z-det-linear}
Let $\underline{\alpha}\in \underline{\mathcal{A}}$ be fixed, let $\mathbb{X}^{\underline \alpha}$ be the unique solution to \eqref{eq:linear_dynamics_intro}--\eqref{eq:definition_of_Z} in $L^2_\boxtimes(\Omega\times I; E)$, and let $\mathbb{Z}^{\underline \alpha}$ be the corresponding aggregate. Then $\mathbb{Z}^{\underline{\alpha}}$ is $\mathbb{P}\boxtimes\lambda$-a.s. equal to a deterministic function in $L^2_\lambda(I;E)$. That is, for some $\widetilde f\in L^2_\lambda(I; E)$:
\begin{equation*}
    \mathbb{P}\boxtimes\lambda\left( \|\mathbb{Z}^{\underline{\alpha}} - f\|_E = 0\right) = 1, \quad f(\omega,x) := \widetilde f(x), \quad (\omega,x) \in \Omega \times I.
\end{equation*}
Furthermore, there exists a unique pair $(X^{\underline\alpha,x})_{x\in I}$ and $(Z^{\underline\alpha,x})_{x\in I}$ of versions of $\mathbb{X}^{\underline{\alpha}}$ and $\mathbb{Z}^{\underline{\alpha}}$, respectively, solving the system \eqref{eq:linear_dynamics_intro}--\eqref{eq:definition_of_Z} for all $x\in I$ in the standard $L^2(\Omega; E)$-sense. Moreover, $Z^{\underline\alpha,\cdot}$ is an everywhere deterministic version of $\mathbb{Z}^{\underline\alpha}$.
\end{theorem}

\begin{remark}
In light of Theorem~\ref{thm:z-det-linear} we can treat the state and aggregate as defined for all $x\in I$ and the aggregate as deterministic:
\begin{equation}
\label{eq:lambda-or-not}
    Z^{\underline\alpha,x}_t 
    = \int_I w(x,y)\mathbb{E}[X_t^{\underline\alpha,y}]\lambda(dy).
\end{equation}
Assumption~\ref{assump:sec3.1}.\ref{assump:sec3.1-ii} alone is not sufficient to let us simplify the aggregate to an integral with respect to the Lebesgue measure $dy$. To wit, the mean state depends on the coefficients. There is a "trade-off" between the measurability assumption on the coefficients and which measure we see in the aggregate integral \eqref{eq:lambda-or-not}. We choose to work under a weaker assumption on the coefficients, and deal with $\lambda$ in \eqref{eq:lambda-or-not}.
\end{remark}

\begin{proof}
For the first statement of the theorem, consider the set of all almost surely deterministic functions in $L^2_\boxtimes(\Omega\times I; E)$
\begin{equation*}
    \widetilde L := \{\widetilde z\in L^2_\boxtimes(\Omega\times I; E) : \mathbb{P}\boxtimes\lambda(\widetilde z = \gamma) = 1 \text{ for some } \gamma(\omega,\cdot) = \widetilde{\gamma}(\cdot) \in L^2_\lambda(I;E)  \}.
\end{equation*}
$\widetilde L$ is a closed subset of $L^2_\boxtimes(\Omega\times I; E)$, hence a complete space. If $\widetilde z\in \widetilde L$, then $([U^{\underline{\alpha}} \widetilde z](x))_{x\in I}$ is an e.p.i. collection of random variables. By ELLN: $[U^{\underline{\alpha}}\widetilde z](\omega,x) = \mathbb{E}[[U^{\underline{\alpha}}\widetilde z](x)]$ for $\mathbb{P}$-a.e. $\omega\in \Omega$, hence $U^{\underline{\alpha}}$ is well-defined as a function from $\widetilde L$ to itself. A similar argument to that of Proposition~\ref{prop:fp-U} yields the contraction property of $U^{\underline{\alpha}}$ restricted to $\widetilde L$. Since $\widetilde L \subset L^2_\boxtimes(\Omega\times I; E)$, the fixed point of $U^{\underline{\alpha}}$ in $\widetilde L$ must be $\mathbb{Z}^{\underline{\alpha}}$, the unique fixed point in $L^2_\boxtimes(\Omega\times I; E)$.

To construct the desired version of $\mathbb{Z}^{\underline{\alpha}}$, we define for a fixed $x\in I$ the mapping $\Phi^{x}$ by as the Bochner integral
\begin{equation*}
   \Phi^{x}(f) := \int_I w(x,y)\mathbb{E}[f^y]\lambda(dy)
\end{equation*}
for any $E$-valued, $\mathbb{P}\boxtimes\lambda$-square integrable, and $\mathcal{F}\boxtimes \mathcal{I}$-measurable function $f$. Clearly, $\Phi^{x}$ is constant over every equivalence class in $L^2_\boxtimes(\Omega\times I; E)$. Denote the element corresponding to $\mathbb{X}^{\underline{\alpha}}$ by $Z^{\underline{\alpha},x}$. For all $x\in I$, $Z^{\underline{\alpha},x} \in E$ by Lemma~\ref{lemma:extended-W}, and it is deterministic. To prove that $(Z^{\underline{\alpha},x})_{x\in I}$ is the desired version of the aggregate, we must simultaneously construct a version of $\mathbb{X}^{\underline{\alpha}}$. Let
\begin{equation*}
    X^{\underline{\alpha},x}_t := \xi^x + \int_0^t \left( a(x)X^{\underline{\alpha},x}_s + b(x)\alpha^x_s + c(x)Z^{\underline{\alpha},x}_s \right) ds + B^x_t,\quad (t,x)\in [0,T]\times I. 
\end{equation*}

Then $t\mapsto X^{\underline{\alpha},x}_t$ is continuous $\mathbb{P}$-a.s. Moreover,
$ \int_I\|Z^{\underline{\alpha},x}\|_E^2\lambda(dx) \leq C\mathbb{E}^\boxtimes\big[\|\mathbb{X}^{\underline{\alpha}}\|_E^2\big]$ so $x \mapsto Z^{\underline{\alpha},x}$ is $\lambda$-square integrable and it follows that $X^{\underline{\alpha}}$ is $\mathbb{P}\boxtimes\lambda$-square integrable. Using this together with admissibility of $\underline{\alpha}$ and Assumption~\ref{assump:sec3.1}, we get that $X^{\underline{\alpha},x} \in L^2(\Omega; E)$ for all $x\in I$.
It remains only to show that $(Z^{\underline{\alpha},x})_{x\in I}$ and $(X^{\underline{\alpha},x})_{x\in I}$ are members of the equivalence classes $\mathbb{Z}^{\underline{\alpha}}$ and $\mathbb{X}^{\underline{\alpha}}$, respectively, since that implies
$Z^{\underline{\alpha},x} = \int_I w(x,y)\mathbb{E}[X^{\underline{\alpha},y}]\lambda(dy)$.
Let $\chi$ and $z$ be arbitrary members of $\mathbb{X}^{\underline{\alpha}}$ and $\mathbb{Z}^{\underline{\alpha}}$, respectively. Then
\begin{equation*}
    \chi_t = \xi + \int_0^ta(\cdot)\chi_s + b(\cdot)\alpha_s + c(\cdot)z_sds + B_t,\quad t\in[0,T],\ \mathbb{P}\boxtimes\lambda\text{-a.s.},
\end{equation*}
and therefore
\begin{equation}
\label{eq:setup-for-gronwall}
    X_t^{\underline{\alpha}} - \chi_t = \int_0^ta(\cdot)\left(X_s^{\underline{\alpha}} - \chi_s\right)ds + \int_0^t c(\cdot) (Z^{\underline{\alpha}}_s - z_s)ds,\ t\in[0,T],\ \mathbb{P}\boxtimes\lambda\text{-a.s.}
\end{equation}
Since $\mathbb{Z}^{\underline{\alpha}}$ is the fixed point to $U^{\underline\alpha}$, $z^\cdot = \int_Iw(\cdot,y)\mathbb{E}[\chi^y]\lambda(dy)$ $\mathbb{P}\boxtimes\lambda$-a.s. for $\chi\in \mathbb{X}^{\underline{\alpha}}$. By construction, $Z^{\underline{\alpha}} = z$ $\mathbb{P}\boxtimes\lambda$-a.s., implying that $Z^{\underline{\alpha}} \in \mathbb{Z}^{\underline{\alpha}}$. Gronwall's lemma applied to \eqref{eq:setup-for-gronwall} then yields $\mathbb{P}\boxtimes\lambda(X^{\underline{\alpha}} = \chi) = 1$, and hence $X^{\underline{\alpha}} \in \mathbb{X}^{\underline{\alpha}}$ which concludes the proof. 
\end{proof}

\subsection{Stochastic maximum principle}
\label{sec:pmp}

We consider a finite horizon stochastic differential game where the cost incurred by player $x\in I$ is composed of a running cost, here modeled as a function $f^x: \mathbb{R}\times\mathbb{R}\times\mathbb{R}\rightarrow\mathbb{R}$ of player $x$'s state, control, and aggregate; and a terminal cost given by a function $h^x:\mathbb{R}\times\mathbb{R}\rightarrow\mathbb{R}$ of player $x$'s terminal state and aggregate. The expected cost when using the admissible strategy $\beta\in \mathcal{A}(x)$ while $\lambda$-a.e. other player uses the strategy profile $\underline{\alpha}$ is $J^x : \mathcal{A}(x) \times \underline{\mathcal{A}} \rightarrow \mathbb{R}$ defined by:
\begin{equation}
\label{eq:def-Jx-beta-alpha}
     J^x(\beta;\underline{\alpha}) :=
     \mathbb{E}
     \Big[
        \int_0^T f^x\big(X^{(\underline{\alpha}^{-x},\beta),x}_t, \beta_t, Z^{\underline{\alpha},x}_t\big)dt 
        + 
        h^x\big(X^{(\underline{\alpha}^{-x},\beta),x}_T, Z^{\underline{\alpha},x}_T\big)
     \Big],
 \end{equation}
where we adopt the notation $(\underline{\alpha}^{-x},\beta)$ for a strategy profile
 \begin{equation*}
     (\underline{\alpha}^{-x},\beta)^y
     =
     \begin{cases}
     \underline{\alpha}^y, &\text{if }y\neq x
     \\
     \beta, &\text{if }y=x.
     \end{cases},\quad
     \underline{\alpha}\in\underline{\mathcal{A}},\ x\in I,\ \beta\in\mathcal{A}(x).
 \end{equation*}
The strategy profile $(\underline{\alpha}^{-x},\beta)$ is admissible and $\mathbb{P}\boxtimes\lambda$-a.e. equal to $\underline\alpha$, so $Z^{(\underline{\alpha}^{-x},\beta), x} = Z^{\underline{\alpha},x}$. 
Other players' actions appear in player $x$'s cost only implicitly through $Z^{\underline{\alpha},x}$, which is deterministic by Theorem~\ref{thm:z-det-linear}, and unchanged if any one specific player changes control. In light of this, we write $J^x(\beta;\underline{\alpha})$ as $\mathcal{J}^x(\beta;Z^{\underline{\alpha},x})$ for a function $\mathcal{A}(x) \times L^2(I; E) \ni (\alpha,z) \mapsto \mathcal{J}^x(\alpha;z) \in \mathbb{R}$.

In classical game theory, a strategy profile such that no player can do better by unilaterally changing strategy is called a Nash equilibrium. In the present context, we adopt the following notion of equilibrium:
 
\begin{definition}
An admissible strategy profile $\underline{\hat{\alpha}}$ is a graphon game Nash equilibrium if 
 \begin{equation*}
    \mathcal{J}^x(\hat{\alpha}^x;Z^{\underline{\hat{\alpha}};x}) \leq \mathcal{J}^x(\beta;Z^{\underline{\hat{\alpha}};x}),\quad \beta \in \mathcal{A}(x),\ x\in I.
 \end{equation*}
\end{definition}

In differential games, equilibria are often characterized with Pontryagin's maximum principle. Proposition~\hyperref[prop:nec]{\ref{prop:nec}} below provides, player by player, a necessary condition  and a sufficient condition for an admissible strategy profile to be a graphon game Nash equilibrium. The proof follows standard lines for the stochastic maximum principle, see for example \cite{yong1999stochastic}.

\begin{proposition}
\label{prop:nec}
Assume that for all $x\in I$, the functions $f^x$ and $h^x$ are measurable and that for all $(x,u,z)\in I\times \mathbb{R}\times\mathbb{R}$, $\chi \mapsto f^x(\chi,u,z)$ and $\chi \mapsto h^x(\chi,z)$ are differentiable. Then if it exists, a Nash equilibrium $\underline{\hat{\alpha}}$ must satisfy
\begin{equation}
\label{eq:nec_cond_for_alpha}
    \hat{\alpha}^x_t \in \underset{u\in \mathbb{R}}{\arg\inf}\, H^x(t,\hat{X}^x_t,u, p^x_t),\quad
    \text{a.e. } t\in[0,T],\ \mathbb{P}\text{-a.s},
\end{equation} 
for each $x\in I$, with $(\hat{X}^x,p^x,q^x)$ solving the Hamiltonian system
\begin{equation}
\label{eq:nec_cond}
\begin{cases}
&d\hat{X}^x_t 
= 
\partial_p H^x(t,\hat{X}^x_t, \hat{\alpha}^x_t, p^x_t)dt + dB^x_t,\qquad \hat{X}_0^x = \xi^x,
\\
&dp^x_t 
= 
-\partial_\chi H^x(t,\hat{X}^x_t, \hat{\alpha}^x_t, p^x_t)dt + q^x_tdB^x_t,\quad 
p^x_T = \partial_\chi h^x(\hat{X}_T^x, \hat{Z}^x_T),
\end{cases}
\end{equation}
where $H^x : [0,T]\times\mathbb{R}\times\mathbb{R}\times\mathbb{R} \rightarrow \mathbb{R}$ is the Hamiltonian of player $x$,
\begin{equation*}
        H^x(t,\chi,u,p) = f^x(\chi,u,\hat{Z}^x_t) + (a(x)\chi + b(x)u + c(x)\hat{Z}^x_t)p,
\end{equation*}
and $\hat{Z}^x_t$ is the aggregate corresponding to the class $\hat{X}^\cdot_t$. 

If, in addition $(\chi,u) \mapsto f^x(\chi,u,z)$ is jointly convex and $\chi\mapsto h^x(\chi,z)$ is convex for $z\in \mathbb{R}$, then any $\underline{\alpha}\in\underline{\mathcal{A}}$ satisfying~\eqref{eq:nec_cond_for_alpha} for every $x \in I$ is a Nash equilibrium.  
\end{proposition}
The system \eqref{eq:nec_cond} is \textit{a priori} only formally written. The next section is devoted to proving existence and uniqueness of solutions to \eqref{eq:nec_cond} in the linear-quadratic setting, making the argument of Proposition~\ref{prop:nec} rigorous there.
In particular, existence and uniqueness of the Nash equilibrium can be obtained if existence and uniqueness of the solution to the forward-backward system~\eqref{eq:nec_cond} holds, as will be the case in the sequel.

\subsection{Existence and uniqueness of solutions to the Hamiltonian system}
\label{sec:FBSDE}

In this section we analyze the FBSDE \eqref{eq:nec_cond} for the linear quadratic case. Let $C_f : I \rightarrow \mathbb{R}^{3\times 3}_{\text{sym}}$ and $C_h: I \rightarrow \mathbb{R}^{2\times 2}_{\text{sym}}$ be bounded and $\mathcal{I}$-measurable and let from now on $f^x$ and $h^x$ be the mappings
\begin{equation}
\label{eq:quadratic_costs}
    f^x(\mathsf u) = \frac{1}{2}\mathsf u^*
                    C_f(x)
                    \mathsf u,\quad 
    h^x(\mathsf v) = \frac{1}{2}\mathsf v^*
            C_h(x)
            \mathsf v,\quad \mathsf u \in \mathbb{R}^3, \mathsf v \in \mathbb{R}^2.
\end{equation}
Whenever $[C_f(x)]_{22}>0$, the unique minimizer of $\mathbb{R}\ni\alpha \mapsto H^x(t,\chi, \alpha, p)$ is
\begin{equation}
\label{eq:foc_}
     \frac{1}{[C_f(x)]_{22}}\Big(-[C_f(x)]_{12}\chi - [C_f(x)]_{32}\hat{Z}^x_t - b(x)p\Big).
\end{equation} 
Plugging \eqref{eq:quadratic_costs}--\eqref{eq:foc_} into \eqref{eq:nec_cond} yields, formally, the linear FBSDE system
\begin{equation}
\label{eq:fbsde_ex_game}
\begin{aligned}
&d\begin{bmatrix}
\hat{X}^x_t
\\
p^x_t
\end{bmatrix}
=
\Gamma(x)
\begin{bmatrix}
\hat{X}^x_t
\\
p^x_t
\end{bmatrix}
+
\Gamma_Z(x)\hat{Z}^x_t
dt 
+
\begin{bmatrix}
1
\\
q^x_t
\end{bmatrix}
dB^x_t,
\\
&\hat{Z}^x_t = \int_I w(x,y)\mathbb{E}[\hat X^y_t]\lambda(dy),\quad 
\hat{X}^x_0 = \xi^x,\quad 
p^x_T
=
\Gamma_T^*(x)
\begin{bmatrix}
\hat{X}^x_T
\\
\hat{Z}^x_T
\end{bmatrix},\quad t\in[0,T],\ x\in I.
\end{aligned}
\end{equation}
where $\Gamma(x) \in \mathbb{R}^{2\times 2}$ and  $\Gamma_Z(x),\Gamma_T(x) \in \mathbb{R}^{2\times 1}$ depend on the coefficient functions $a,b,c$ and the cost matrices $C_f, C_h$. Their exact form is presented in the appendix. To ensure solvability of \eqref{eq:fbsde_ex_game} we enforce the following condition.

\begin{assumption}
\label{cond:riccati}
For all $x\in I$: (i) $\Gamma_{12}(x) \neq 0$; (ii) $-[C_h(x)]_{11}\Gamma_{12}(x) \geq 0$; (iii) $-\Gamma_{12}(x)\Gamma_{21}(x) > 0$; (iv) $[C_f(x)]_{22} > 0$.
\end{assumption}

\begin{theorem}
\label{thm:fbde}
There exists a unique solution $(\hat{X}, p,q) \in L^2_\boxtimes(\Omega\times I; E) \times L^2_\boxtimes(\Omega\times I; E)\times L^2_\boxtimes(\Omega\times I; L^2([0,T]))$ to the FBSDE 
\eqref{eq:fbsde_ex_game} for any finite time horizon $T>0$.
\end{theorem}

From the $L^2$-solution granted by the theorem we can extract one version that satisfies \eqref{eq:fbsde_ex_game} for all $x\in I$ by replicating the argument from Theorem~\ref{thm:z-det-linear}.

While uniqueness follows by quite standard arguments, existence requires a more in-depth analysis. We argue in the appendix along the following lines: after making the ansatz that $p^x$ is linear in $\hat{X}^x$, we get short time existence following a fixed point argument similar to that of Theorem~\ref{thm:z-det-linear}, then we extend the existence to the whole time horizon $[0,T]$ with the induction method, see, \textit{e.g.}, \citep[Sec. 4.1.2]{carmona2018probabilistic}.

\section{Connection with $N$-player Network Games}
\label{sec:convergence}

In this section we show how certain finite player games where  players interact through a randomly-weighted graph can be approximated by a graphon game.
Let $I^\infty$ be the infinite 
countable Cartesian product of copies of $I$. Throughout this chapter, $x^\infty = (x_k)_{k=1}^\infty$ denotes a generic sequence in $I^\infty$ and $N\in \mathbb{N}$. For a fixed $x^\infty$, let $\mathcal{A}_N(x^\infty)$ denote the set of $\bigvee_{\ell=1}^N \mathbb{F}^{x_\ell}$-progressively measurable functions in $L^2([0,T]\times\Omega)$ (from now on referred to as the set of admissible controls in the $N$-player game). Consider the $N$-player game where player $k$ picks a control $\alpha^{k,N} \in \mathcal{A}_N(x^\infty)$ so as to minimize
\begin{equation*}
    J^{k,N}(\alpha^{k,N}; \alpha^{-k,N}) 
    := 
    \mathbb{E}\Big[\int_0^T f^{x_k}(X^{k,N}_t, \alpha^{k,N}_t, Z^{k,N}_t)dt 
    + h^{x_k}(X^{k,N}_T, Z^{k,N}_T)\Big]
\end{equation*}
and the player states $(X^{k,N})_{k=1}^N$ are subject to the dynamics
\begin{align*}
    dX^{k,N}_t &= \big(a(x_k)X^{k,N}_s + b(x_k)\alpha^{k,N}_t + c(x_k)Z^{k,N}_t\big)dt + dB^{x_k}_t, 
    \quad 
    X^{k,N}_0 = \xi^{x_k},
    \\
    Z^{k,N}_t &:= \frac{1}{N}\sum_{\ell= 1}^N w(x_k, x_\ell)X^{\ell,N}_t,\quad k=1,\dots, N,\ t\in[0,T].
\end{align*}

Assumption~\ref{assump:sec3.1} and \ref{cond:riccati} are in force as well as the additional assumptions on the cost functions introduced in Section~\ref{sec:FBSDE}.
With an approach very similar to that of Proposition~\ref{prop:nec}, we can derive necessary conditions for Nash equilibria in this $N$-player game. Given $x^\infty\in I^\infty$, a Nash equilibrium in the $N$-player game satisfies $\mathbb{P}$-a.s., a.e.-$t$, for all $k=1,\dots, N$,
\begin{equation*}
    \hat{\alpha}^{k,N}_t 
    = 
    -\frac{1}{[C_f(x_k)]_{22}}\Big(b(x_k)p^{k k,N}_t + [C_f(x_k)]_{21}X^{k,N}_t + [C_f(x_k)]_{23}Z^{k,N}_t\Big),
\end{equation*}
where the $N$ state- and costate variables $(X^{k,N},p^{kk,N})_{k=1}^N$ solve the forward-backward system
\begin{align*}
    &d\begin{bmatrix}
    X^{k,N}_t
    \\
    p^{kk,N}_t
    \end{bmatrix}
    =
    \left(\Gamma(x_k)\begin{bmatrix}
    X^{k,N}_t
    \\
    p^{kk,N}_t
    \end{bmatrix}
    +
    \Gamma_Z(x_k) Z^{k,N}_t
    +
    \begin{bmatrix}
    0
    \\
    -\frac{1}{N}\sum_{\ell=1}^N c(x_\ell) w(x_k,x_\ell)p^{k\ell,N}_t
    \end{bmatrix}\right)dt 
    \\
    &\hspace{2cm} + 
    \begin{bmatrix}
    dB^{x_k}
    \\
    \sum_{\ell=1}^N q^{kk\ell,N}_tdB^{x_\ell}
    \end{bmatrix},\quad Z^{k,N}_t := \frac{1}{N}\sum_{\ell=1}^N w(x_k,x_\ell)X^{\ell,N}_t,\quad t\in[0,T],
    \\
    &X^{k,N}_0 = \xi^{x_k},\quad p^{kk,N}_T = \Gamma^*_T(x_k)\begin{bmatrix}
    X^{k,N}_T
    \\
    Z^{k,N}_T
    \end{bmatrix},\quad k=1,\dots, N,
\end{align*}
and are coupled with the off-diagonal costate variables $(p^{kh}, (q^{kh\ell})_{\ell=1}^N)_{k\neq h=1}^N$, solving the backward system below where $\bar \Gamma, \bar\Gamma_T : I\rightarrow \mathbb{R}^{2\times 1}$ and $\bar\Gamma_Z : I \rightarrow \mathbb{R}$:
\begin{align*}
    &dp^{kh,N}_t 
    = \frac{w(x_k,x_h)}{N}\Big(\bar\Gamma(x_k)
    \begin{bmatrix}
    X^{k,N}_t
    \\
    p^{kk,N}_t
    \end{bmatrix} 
    +
    \bar\Gamma_Z(x_k)Z^{k,N}_t\Big)dt
    \\
    &\hspace{2cm}- 
    \Big(\frac{1}{N}\sum_{\ell=1}^N c(x_\ell)w(x_h,x_\ell)p^{k\ell,N}_t + a(x_h) p^{kh,N}_t \Big)dt 
    + \sum_{\ell=1}^N q^{kh\ell,N}_tdB^{x_\ell}_t,
    \\ 
    &p_T^{kh,N} = \frac{w(x_k,x_h)}{N}\bar \Gamma^*_T(x_k)\begin{bmatrix}
    X^{k,N}_T
    \\
    Z^{k, N}_T
    \end{bmatrix} ,\quad 1\leq h\neq k\leq N,\ t\in [0,T].
\end{align*}
This finite FBSDE system is linear with bounded coefficients. With additional assumptions (to insure solvability of a matrix Riccati equation), it can be analyzed with the induction approach as was done in Theorem~\ref{thm:fbde}. We abstain from presenting the argument, instead we move on assuming that the solution to the coupled system above is well-defined and of sufficient regularity for the following analysis.

\subsection{Propagation of chaos}

Assume that the conditions from earlier sections are satisfied and denote by $(X^x,p^x,q^x)_{x\in I}$ the solution to the graphon game FBSDE.  By comparing the graphon game FBSDE and the $N$-player game FBSDE along a specific sequence $x^\infty$, we obtain a propagation of chaos-type result. Let 
\begin{align*}
    &\Delta(x^\infty, N)
    :=  
    \\
    &\max_{1\leq k\leq N}
    \Big(\mathbb{E}\Big[\sup_{t\in[0,T]}\big(|X^{k,N}_t - X^{x_k}_t|^2 \hspace{-3pt}
    + 
    \hspace{-3pt}
    |p^{k k,N}_t - p^{x_k}_t|^2\big) \Big]
    + \sup_{t\in[0,T]}\mathbb{E}\Big[|Z^{x_k,N}_t - Z^{x_k}_t|^2\Big]\Big).
\end{align*}
To derive statistical estimates for the rate at which $\Delta(\cdot, N)$ tends to zero, we introduce the iteratively completed infinite product space $(I^\infty, \bar{\mathcal{I}}^\infty,\bar\lambda^\infty)$, defined in line with \citep{hammond2006essential}. The $\sigma$-algebra $\bar{\mathcal{I}}^\infty$ extends $\mathcal{I}^\infty$ by including sets that are "iteratively null". An important feature of this setup that we will frequently use is that $(B^x, \xi^x)_{x\in x^\infty}$ are mutually independent for $\bar\lambda^\infty$-a.e. $x^\infty \in I^\infty$ \citep[Theorem 1]{hammond2016one}. We strengthen the assumptions on the graphon to prove the convergence.

\begin{assumption}
\label{cond:poc-and-cont}
    The mapping $I\ni x\mapsto w(x,y)\in \mathbb{R}$ is $1/2$-Hölder continuous, uniformly in $y\in I$.
\end{assumption}

\begin{proposition}
\label{prop:poc-conv-rate}
Let Assumption~\ref{assump:sec3.1} and \ref{cond:riccati} hold. Then
\begin{equation*}
    \Delta(x^\infty, N) \xrightarrow[N \to +\infty]{} 0,\qquad \bar\lambda^\infty\text{-a.e. } x^\infty\in I^\infty.
\end{equation*}
If furthermore Assumption~\ref{cond:poc-and-cont} also holds, then for all $\varepsilon > 0$ there exists a $N_\varepsilon : I^\infty \to \mathbb{N}$ such that
\begin{equation*}
     \bar\lambda^{\infty}\Big(\Delta(x^\infty, N) \leq \frac{(C+\varepsilon)^2\log\log N}{N},\ N \geq N_\varepsilon(x^\infty) \Big) = 1,
\end{equation*}
 where $C$ is a finite positive constant depending only on $T$ and the graphon $w$.
\end{proposition}

\begin{remark}
Under a different set of conditions, the rate of convergence can be shown to be $1/N$ with high probability. Indeed, if the graphon is of rank $1$ and $a,b,c,C_f,C_h$ are constant, it is possible to prove under some conditions on these constants and the eigenvalues of the graphon operator, that
\begin{equation*}
     \bar\lambda^{\infty}\Big(\Delta(x^\infty, N) \leq \frac{C}{N}\Big) > 1-2e^{-2C^2},
\end{equation*}
for all $C \geq \bar{C} > 0$, with $\bar C$ a finite constant depending only on $T$ and the graphon $w$. Importantly, this set of conditions does not require continuity of the graphon, and it covers the important class of piecewise constant graphons, associated with the family of stochastic block models rich in applications.
\end{remark}

\subsection{Convergence and approximation of $N$-player game Nash equilibrium}

The propagation of chaos type result contained in Proposition \ref{prop:poc-conv-rate} directly yields two $\bar\lambda^{\infty}$ - a.s. results relating the Nash equilibria of $N$-player and graphon game: 1) the $N$-player Nash equilibria converge toward the graphon game Nash equilibria; 2) the graphon game Nash equilibria provide approximate $N$-player game Nash equilibria.

Let $C$ be the coefficient in Proposition~\ref{prop:poc-conv-rate} and let $\varepsilon_N := 2C\sqrt{N^{-1}\log\log N}$. Let $\underline{N}$ be the random variable $N_\varepsilon$ from Proposition~\ref{prop:poc-conv-rate} with $\varepsilon = C$. As before, we denote by $(\hat{\alpha}^{k,N})_{k=1}^N$ and $(\hat \alpha^x)_{x\in I}$ the Nash equilibria for the $N$-player game and the graphon game, respectively. 

\begin{proposition}
\label{prop:eps-eq}
The graphon game Nash equilibrium strategy collection $(\hat\alpha^{x_k})_{k=1}^N$ forms an $\varepsilon_N$-Nash equilibrium for the $N$-player game between the players $(x_1,\dots, x_N)$ when $N\geq \underline{N}(x^\infty)$, $\bar\lambda^\infty$-a.s. That is, for all $\beta \in \mathcal{A}_N(x^\infty),\ k=1,\dots, N,\ N\geq\underline{N}(x^\infty)$:
\begin{equation*}
    J^{k,N}(\hat{\alpha}^{x_k}; \bar{\alpha}^{-k,N}) - J^{k,N}(\beta; \bar{\alpha}^{-k,N}) \leq \varepsilon_N,\quad \bar\lambda^\infty\text{-a.e. }x^\infty\in I^\infty,
\end{equation*}
where $\bar{\alpha}^{-k,N} := (\hat{\alpha}^{x_1},\dots, \hat{\alpha}^{x_{k-1}},\hat{\alpha}^{x_{k+1}},\dots, \hat{\alpha}^{x_N})$. 

Moreover, the $N$-player game equilibrium converges componentwise to the graphon game equilibrium and the rate of convergence is uniform and at most $\varepsilon_N$:
\begin{equation*}
    \max_{1\leq k\leq N}\mathbb{E}\Big[\int_0^T|\hat\alpha^{k,N}_t - \hat\alpha^{x_k}_t|^2dt\Big] \leq \varepsilon_N^2,\quad N\geq \underline{N},\ \bar\lambda^\infty\text{-a.e. }x^\infty\in I^\infty.
\end{equation*}
\end{proposition}

\section{Semi-explicit solutions in the case of constant coefficients}
\label{sec:examples}

When the functions $a,b,c, C_f,$ and $C_h$ are constant over $I$, the eigenfunction expansion \eqref{eq:graphon_decomp} facilitates a reformulation of the FBSDE into a countable system of decoupled ODEs. Note that if the graphon is of finite rank, the system is in fact of finite size. This idea was already used in \citep{gao2020lqg}. Here we summarize it, comment on the well-posedness of the resulting ODE system, and numerically solve an example.

First, recall from the proof of Theorem~\ref{thm:fbde} that the FBSDE can, by using the linear ansatz $p^x_t = \eta^x_t\hat{X}^x_t + \zeta^x_t$, be reduced to a forward-backward system for $(\hat X^x, \zeta^x)_{x\in I}$ and a Riccati equation for $\eta^x$:
\begin{equation}
\label{eq:eta-zeta}
    \begin{aligned}
    &\dot{\eta}^x_t = - (\eta^x_t)^2\Gamma_{12}(x) - \eta^x_t(\Gamma_{11}(x) + \Gamma_{21}(x) - \Gamma_{22}(x)) + \Gamma_{21}(x),& &\eta^x_T = \Gamma_{T,1}(x),
    \\
    &d\begin{bmatrix}
    \hat{X}_t^x
    \\
    \zeta^x_t
    \end{bmatrix}
    =
    \begin{bmatrix}
    \Gamma_{11}(x) + \eta^x_t\Gamma_{12}(x) & \Gamma_{12}(x)
    \\
    0 & \Gamma_{22}(x) - \eta^x_t\Gamma_{12}(x)
    \end{bmatrix}
    \begin{bmatrix}
    \hat{X}_t^x
    \\
    \zeta^x_t
    \end{bmatrix}
    dt
    \\
    &\hspace{2cm}+
    \begin{bmatrix}
    \Gamma_{Z,1}(x)
    \\
    \Gamma_{Z,2}(x) - \eta^x_t\Gamma_{Z,1}(x)
    \end{bmatrix}\hat{Z}^x_t
    dt
    +
    \begin{bmatrix}
    1
    \\
    0
    \end{bmatrix}
    dB^x_t,& &x\in I,\ t\in[0,T],
    \\
    &
    \hat{X}^x_0 = \xi^x,\quad
    \zeta^x_T =
    \Gamma_{T,2}(x)\hat{Z}^x_T,\quad 
    \hat{Z}^x_t = \int_I w(x,y)\mathbb{E}[\hat{X}^y_t]\lambda(dy),& &x\in I,\ t\in[0,T].
    \end{aligned}
\end{equation}
The Riccati equation does not depend on $(\hat{X}^x, \zeta^x)$ and can be solved independently. In the case of constant coefficients, the Riccati equation is also independent of $x$ and we then denote its solution $(\eta_t)_{t\in[0,T]}$. The next assumption formalizes the constant-coefficient condition.
\begin{assumption}
The functions $\nu, a, b, c, C_f, C_h$ are constant.
\end{assumption}

The ansatz coefficient $\zeta^x$ and the aggregate $\hat{Z}^x$ solve the coupled forward-backward ODE
\begin{equation}
\label{eq:sys_for_zeta_and_z}
    \begin{aligned}
    &\dot{\zeta}^x_t = \zeta^x_t\left(\Gamma_{22}-\Gamma_{12}\eta_t\right) +\hat{Z}^x_t\left(\Gamma_{Z,2}-\Gamma_{Z,1}\eta_t\right),& &\zeta^x_T = \Gamma_{T,2}\hat{Z}^x_T,
    \\
    &\frac{d}{dt}\hat{Z}^x_t = (\Gamma_{11} + \Gamma_{12}\eta_t)\hat{Z}^x_t + \Gamma_{12}[W\zeta^\cdot_t](x) + \Gamma_{Z,1}[W\hat{Z}^\cdot_t](x),& &\hat{Z}^x_0 = [W\xi^\cdot](x).
    \end{aligned}
\end{equation}
At this point, we use the expansion \eqref{eq:graphon_decomp} and in the eigendirection $\phi_k$, $k\in\mathbb{N}$, we get
\begin{equation}
\label{eq:sys_for_z_and_zeta}
    \begin{aligned}
        &\dot{v}^k_t = v^k_t(\Gamma_{22}-\Gamma_{12}\eta_t) + z^k_t(\Gamma_{Z,2}-\Gamma_{Z,1}\eta_t),& &v^{k}_T = \Gamma_{T,2}z^k_T,
        \\
        &\dot{z}^k_t = \big(\Gamma_{11} + \Gamma_{12}\eta_t + \Gamma_{Z,1}\lambda_k\big)z^k_t + \Gamma_{12}\lambda_k v^{k}_t,& &z^{k}_0 = \lambda_k x^{k}.
    \end{aligned}
\end{equation}
where $v^k_t := \langle \zeta_t, \phi_k\rangle_{\lambda_I}$, $z^k_t := \langle \hat{Z}_t,\phi_k\rangle_{\lambda_I}$, and $x^k := \langle \xi, \phi_k\rangle_{\lambda_I}$. Further analysis of the coupled system \eqref{eq:sys_for_z_and_zeta} with the ansatz $v^k_t = \pi^k_tz^k_t$, where $\pi^k$ is a deterministic function of time to be determined, yields a Riccati equation with time-varying coefficients for $\pi^k$:
\begin{equation}
\label{eq:riccati_for_pi}
\begin{aligned}
    &\dot{\pi}^k_t 
    = 
    -\Gamma_{12}\lambda_k\big(\pi^k_t\big)^2 - \left(\Gamma_{11} + 2\Gamma_{12}\eta_t - \Gamma_{22} + \Gamma_{Z,1}\lambda_k\right)\pi^k_t - (\Gamma_{Z,2} - \Gamma_{Z,1}\eta_t),
    \\
    &
    \pi^k_T 
    =
    \Gamma_{T,2}. 
\end{aligned}
\end{equation}
Existence and uniqueness of a non-blow up solution to \eqref{eq:riccati_for_pi} requires further assumptions on the coefficients. In the appendix, we derive a set of sufficient conditions. With $\pi^k$ at hand it is a simple task to solve \eqref{eq:sys_for_z_and_zeta} for $z^k$ and $v^k$ ($x^k$ is computed with the Exact Law of Large Numbers), and 
\begin{equation*}
        [W\hat{Z}_t](x) = \sum_{k=1}^\infty \lambda_k z^k_t\phi_k(x),\qquad [W\zeta_t](x) = \sum_{k=1}^\infty \lambda_k v^k_t\phi_k(x).
\end{equation*}
The system \eqref{eq:sys_for_zeta_and_z} has been decoupled and rewritten into a countable set of ODEs.

\subsection{Comparison of three graphons' impact on the graphon game Nash equilibrium}

Using the semi-explicit solution outlined above it is easy to numerically solve the graphon game. Let us compare the impact of three graphons on the solution: the constant graphon, the power-law graphon, and the min-max graphon.

If the graphon is constant, $w_C(x,y) = K$, then the graphon game is equivalent to a MFG. Indeed, in this case  $\hat{Z}^x_t = K\int_I\mathbb{E}[\hat X^x_t]\lambda(dx)$ and one can show that $\hat{Z}^x_t = K\mathbb{E}[\hat X^0_t] = K\mathbb{E}[\hat X^y_t]$ for $x,y\in I$, $t\in[0,T]$. The role of each player $x\in I$ is that of the representative player in the MFG. The power-law graphon is defined as $w_{PL}(x,y) := (xy)^{-\gamma}$, $\gamma\in\mathbb{R}$. It has applications in dynamical systems theory, see \textit{e.g.} \citep{medvedev2018kuramoto} where it is used to model coupled systems on scale-free graphs. The power-law graphon induces an operator of rank $1$. If $\gamma \leq 0$, $w_{PL}$ is bounded with eigenvalue $\lambda_{PL} = (1-2\gamma)^{-1}$ and eigenfunction $\phi_{PL}(x) = \sqrt{\lambda_{PL}}x^{-\gamma}$. The min-max graphon $w_{MM}(x, y) = \min(x, y)(1 - \max(x, y))$ is not of finite rank. The orthonormal basis of eigenfunctions is given by $\phi_{k,MM}(x) = \sqrt{2} \sin(\pi k x)$ with corresponding eigenvalues $\lambda_{k,MM} = (\pi k)^{-2}$ for $k \geq 1$, see \citep{avella2018centrality}.

For the sake of presentation we consider a particularly simple case of the linear-quadratic graphon game. The parameter values are chosen for the purpose of visualization of the effects we want to highlight with numerical simulation, but can be taken arbitrarily from within our theoretical assumptions. The cost and state equation of agent $x$ are set to
\begin{equation}
\label{example:cost}
    \frac{1}{2}\mathbb{E}\Big[\int_0^3 \big((\alpha^x_t)^2 + (X^x_t-Z^x_t)^2\big) dt + (X_T^x - Z^x_T)^2\Big]
\end{equation}
and
\begin{equation}
\label{example:dynamics}
\begin{aligned}
    dX^x_t &= \left(-X^x_t + \alpha^x_t + Z^x_t \right)dt + dB^x_t,\quad X^x_0 = \xi^x \sim \text{Normal}(8,1/4),
    \\
    Z^x_t &= \int_Iw(x,y)\mathbb{E}[X^y_t]\lambda(dy),\quad x\in I,\ t\in[0,T].
\end{aligned}
\end{equation}
Here, Assumption~\ref{cond:riccati} holds and if $|\lambda_k|\leq 1$, which is the case for the three graphons we consider, then the time-varying Riccati equation \eqref{eq:riccati_for_pi} does not blow up in finite time. We simulate \eqref{example:cost}--\eqref{example:dynamics} using the semi-explicit solution derived above. For the constant graphon, we set $K=1$, and for the power-law graphon, we set $\gamma = -0.4$. The simulation uses standard discretization techniques to compute the state trajectory for a finite number of values of $x$: $(x_m)_{m = 1}^M$. We set $M = 200$ and sample the indices $(x_m)_{m = 1}^M$ from the uniform distribution over $I$. (The choice of indices here has no effect on the computation of the Nash equilibrium.) 

The simulation results are presented in Figure~\ref{fig:my_label}. In the constant graphon case we expect the state to mean-revert to the mean of the initial conditions, as is the case for the MFG corresponding to this specific setting. The simulation result agrees with this. For the other two graphons however, state trajectories that are tending to $0$ due to reversion to the aggregate; the aggregate tends to zero in the non-constant graphon cases. We also note the player index has an impact on the state trajectory in the non-constant graphon cases. The state trajectories are clearly ordered in space according to index in the power-law graphon case. The index influences the state trajectory also in the min-max graphon case, albeit in a different way and to a smaller extent (almost indistinguishably in the figure but visible in a more detailed simulation; on average, the further from $0.5$ the index is, the steeper the initial descent).
\begin{figure}[h]
    \centerline{
    \includegraphics[width=\textwidth]{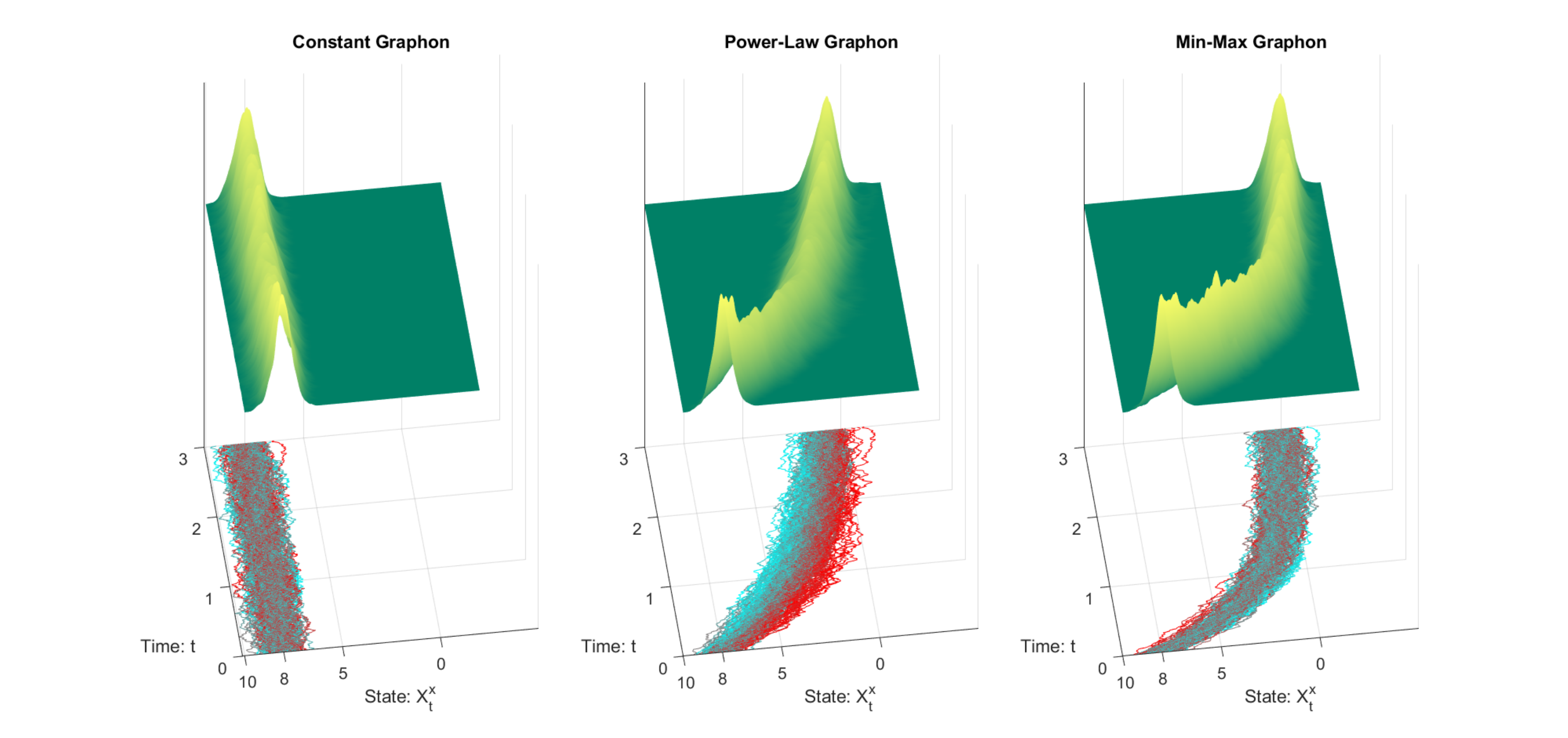}
    }
    \caption{Simulation of the graphon game at equilibrium, evaluating the influence of the graphon type on the solution. The same game is solved with three different aggregate variables, defined with the constant, the power-law, and the min-max graphon. 
\textbf{Top:} The green-yellow landscape visualizes the player population occupation measure over time. Only in the constant graphon case is this equal to the (empirical) distribution of the representative agent's state in the MFG.
\textbf{Bottom:} The blue-to-red graphs in the plane are simulated player state trajectories and the color indicates the corresponding player index. Interpolating between red and blue, higher indices are plotted in the blue while lower indices more in the red for the constant and power-law graphons. In the min-max graphon case, the more red the trajectory, the closer to $0.5$ is the corresponding index.}
    \label{fig:my_label}
\end{figure}

\section{Summary and Concluding Remarks}
\label{sec:conclusion}
In this paper, we formulate and analyze a stochastic differential game for a continuum of players subject to idiosyncratic random shocks modeled as a continuum of essentially pairwise independent Brownian motions, and interacting over a network structure given by a graphon. Using the framework of Fubini extensions and the Exact Law of Large Numbers, we demonstrate how to control the system. We characterize the Nash equilibria with the Pontryagin stochastic maximum principle and we prove a form of propagation of chaos providing an approximation of the graphon game by finite player network games (and vice versa).

We choose to present our theoretical results in the linear-quadratic setting in order to provide a complete analysis without too many technical assumptions.
The linear-quadratic case is the standard model in stochastic differential game theory, generally acknowledged as important and with a huge number applications,  somewhat restrictive though that may be. Some of our results hold true under weaker assumptions, let us briefly elaborate on the possible extensions of some of the results of the paper.

Theorem~\ref{thm:z-det-linear} generalizes beyond the linear model. With minor modifications to the proofs, the same conclusions hold when
\begin{equation*}
    dX^{\underline{\alpha},x}_t = \beta\big(x,X^{\underline{\alpha},x}_t ,\alpha^x_t, Z^{\underline{\alpha},x}_t\big)dt + dB^x_t,\quad X^x_0 = \xi^x,
\end{equation*}
and $Z^{\underline{\alpha},x}_t = [WK(X_t^{\underline{\alpha},\cdot})](x)$ under the assumption that $\beta$ and $K$ are $\mathcal{I}\otimes (\mathcal{B}(\mathbb{R}))^3$- and $\mathcal{B}(\mathbb{R})$-measurable, respectively; $\sup_{x\in I}|\beta(x,0,0,0)|<\infty$; and $(\chi,a,z) \mapsto \beta(x,\chi,a,z) + K(\chi)$ is Lipschitz continuous, uniformly in $x\in I$. With the additional assumption that $\chi \mapsto \beta(x,\chi, a,z)$ is differentiable, the maximum principle Proposition~\ref{prop:nec} can be lifted to the same level of generality. That said, a generalization of Theorem \ref{thm:fbde} would not be as straightforward. The induction approach would require an analysis of the master equation for the control problem, and prove that the gradient of its solution is uniformly Lipschitz continuous, which is beyond the scope of this study.

With regards to the convergence results, we expect the approach developed here to extend to other sequences interaction networks (than the one defined by evaluating the graphon on the $N = 1,2,\dots$ first elements of a fixed index sequence $x^\infty$) as long as they converge to a limiting graphon, \textit{e.g.}, in the cut-distance (see \cite{lovasz2012large}). 
Indeed, it is already known in some cases that various finite player game structures are approximated by the same MFG \cite{delarue2017mean}. It would follow that the probabilistic limits studied in Section~\ref{sec:convergence} extend, possibly with some slight modifications to the quantitative bounds, to sequences of games with random interaction graphs generated with the sampling method for graphons, as developed in \cite{lovasz2012large}.

\begin{appendix}
\section{Proofs}
\label{app:proof}

\subsection{Proof of Proposition~\ref{prop:fp-U}}

We drop the superscript $\underline\alpha$ since the strategy profile does not change throughout the proof.  By Lemma~\ref{lemma:well-posedness-of-X}, $\mathbb{X}^z\in L^2_\boxtimes(\Omega\times I; E)$ so by Lemma~\ref{lemma:extended-W}, the mapping $U$ is well defined. We turn to the contraction property.

To prove that $U$ is a strict contraction, let $z, \tilde{z} \in L^2_{\boxtimes}(\Omega\times I; E)$. By Gronwall's inequality and the boundedness of the graphon, 
\begin{equation}
\label{eq:contration_proof}
    \begin{aligned}
    \mathbb{E}^\boxtimes\big[\|\left[Uz\right] - [U\tilde{z}]\|^2_E\big]
    \leq
    C\int_0^T\mathbb{E}^\boxtimes\big[\sup_{s\in[0,t]}\big|z_s - \widetilde{z}_s\big|^2\Big]dt,
    \end{aligned}
\end{equation}
where $C>0$ is a finite constant depending only on $T$, $\|W\|_2$, and the coefficient bound. Iterating the inequality
\eqref{eq:contration_proof} and making use of the fact that
\begin{equation*}
    \mathbb{E}^{\boxtimes}\Big[\sup_{s\in[0,t]}|z_s|^2\Big] \leq
    \mathbb{E}^{\boxtimes}\big[\|z_s\|^2_E\big],\quad t\in [0,T],\ z\in L^2_{\boxtimes}(\Omega\times I; E),
\end{equation*}
we have
\begin{equation*}
    \mathbb{E}^{\boxtimes}\Big[\big\|[(U^N)z] - [(U^N)\widetilde{z}]\big\|_E^2\Big] \leq \frac{(CT)^N}{N!}\mathbb{E}^{\boxtimes}\Big[\|z - \tilde{z}\|_E^2\Big],\quad N\in\mathbb{N}.
\end{equation*}
Hence, for some $N\in\mathbb{N}$, $(U^N)$ is a contraction mapping on from  $L^2_{\boxtimes}(\Omega\times I; E)$ to itself. The existence of a unique fixed point to $U$ in the Banach space $L^2_{\boxtimes}(\Omega\times I; E)$ then follows by the Banach fixed-point theorem for iterated mappings, see \textit{e.g.} \citep{bryant1968remark}.

\subsection{Proof of Theorem~\ref{thm:fbde}}
\label{proof:thm-fbsde}

The coefficient matrices in \eqref{eq:fbsde_ex_game} are given in terms of $a,b,c,C_f$ and $C_h$ as follows:
\begin{align*}
    &\Gamma := 
    \begin{bmatrix}
    a - \frac{b[C_f]_{12}}{[C_f]_{22}} & -\frac{b^2}{[C_f]_{22}}
    \\
    \frac{[C_f]_{12}^2}{[C_f]_{22}}- [C_f]_{11} & \frac{b[C_f]_{12}}{[C_f]_{22}}- a
    \end{bmatrix},\quad
    \Gamma_Z := 
    \begin{bmatrix}
    c - \frac{b [C_f]_{32}}{[C_f]_{22}}
    \\
    \frac{[C_f]_{12}[C_f]_{32}}{[C_f]_{22}}-[C_f]_{12} 
    \end{bmatrix},\quad
    \Gamma_T := 
    \begin{bmatrix}
        [C_h]_{11}
        \\
        [C_h]_{12}
    \end{bmatrix}.
\end{align*}
\textit{Step 0: Uniqueness in the Fubini extension}\\
Assume that $(X, p, q)$ and $(\widetilde X, \widetilde p, \widetilde q)$ are solutions to the FBSDE \eqref{eq:fbsde_ex_game} in the sense that \eqref{eq:fbsde_ex_game} is satisfied $\mathbb{P}\boxtimes\lambda$-a.s. and $\mathbb{E}^\boxtimes\big[ \|X\|_E^2 + \|p\|_E^2 + \int_0^T|q_t|^2dt\big] < \infty$. 
Uniqueness, \textit{i.e.}, that 
\begin{equation*}
    \mathbb{E}^\boxtimes\Big[\|p - \widetilde p\|_E^2 
    + \|X - \widetilde X\|^2_E
    +
    \int_0^T |q_t - \widetilde q_t|^2dt
    \Big] = 0
\end{equation*}
can be proven along standard lines of proof.
\\~\\
\textit{Step 1: An ansatz for $p^x$}\\
We will look for a solution defined with the following ansatz: for each $x\in I$ there exists differentiable and deterministic mappings $t\mapsto \eta^x_t$ and $t\mapsto \zeta^x_t$ such that
\begin{equation}
\label{eqa:ansatz}
    p^x_t = \eta^x_t \hat{X}^x_t + \zeta^x_t.
\end{equation}
Plugging \eqref{eqa:ansatz} into~\eqref{eq:fbsde_ex_game} and matching terms
we obtain $\eta^x = q^x$ and the following system for $(\hat{X}^x,\zeta^x,\eta^x; x\in I)$:
\begin{equation}
\label{eq:the_first_eta-zeta}
    \begin{aligned}
    &\dot{\eta}^x_t = - (\eta^x_t)^2\Gamma_{12}(x) - \eta^x_t(\Gamma_{11}(x) + \Gamma_{21}(x) - \Gamma_{22}(x)) + \Gamma_{21}(x),& &\eta^x_T = \Gamma_{T,1}(x),
    \\
    &d\begin{bmatrix}
    \hat{X}_t^x
    \\
    \zeta^x_t
    \end{bmatrix}
    =
    \begin{bmatrix}
    \Gamma_{11}(x) + \eta^x_t\Gamma_{12}(x) & \Gamma_{12}(x)
    \\
    0 & \Gamma_{22}(x) - \eta^x_t\Gamma_{12}(x)
    \end{bmatrix}
    \begin{bmatrix}
    \hat{X}_t^x
    \\
    \zeta^x_t
    \end{bmatrix}
    dt
    \\
    &\hspace{2cm}+
    \begin{bmatrix}
    \Gamma_{Z,1}(x)
    \\
    \Gamma_{Z,2}(x) - \eta^x_t\Gamma_{Z,1}(x)
    \end{bmatrix}\hat{Z}^x_t
    dt
    +
    \begin{bmatrix}
    1
    \\
    0
    \end{bmatrix}
    dB^x_t,& &x\in I,\ t\in[0,T],
    \\
    &
    \hat{X}^x_0 = \xi^x,\quad
    \zeta^x_T =
    \Gamma_{T,2}(x)\hat{Z}^x_T,\quad 
    \hat{Z}^x_t = \int_I w(x,y)\mathbb{E}[\hat{X}^y_t]\lambda(dy),& &x\in I,\ t\in[0,T].
    \end{aligned}
\end{equation}
The Riccati equation for $\eta^x$ in \eqref{eq:the_first_eta-zeta} does not depend on the other variables and can be solved independently. Furthermore, under Assumption~\ref{assump:sec3.1} and \ref{cond:riccati} it has a unique solution $(\eta^x_t)_{t\in [0,T]}$ for all $x\in I$ 
and $\sup_{(t,x)\in[0,T]\times I}|\eta^x_t| < \infty$, see for example 
\citep[Sec. 2.4.1]{carmona2018probabilistic}. Thus, to prove existence of a solution to~\eqref{eq:the_first_eta-zeta} it is sufficient to study the forward-backward system for $(\hat{X}, \zeta)$, which is the subject matter of the next steps.
\\~\\
\textit{Step 2: Unique solvability of \eqref{eq:the_first_eta-zeta} for short time horizons}\\
If we fix a collection of aggregates $\hat{Z}^x\in E$, $x\in I$, then $\zeta^x$ and subsequently $\hat{X}^x$ can be solved explicitly for all $x\in I$. This "decoupling" property of the aggregate provides us with a simple proof of short time existence and uniqueness. By a fixed-point argument there exists a unique solution $(\hat{X}^x,\zeta^x)$ to \eqref{eq:the_first_eta-zeta} in $L^2_\boxtimes(\Omega\times I; E)\times L^2(I;E)$ when $T$ is small enough.
\\~\\
\textit{Step 3: Setting the stage for the induction approach}\\
Inspired by the induction approach, described in detail in \citep[Sec. 4.1.2.]{carmona2018probabilistic}, we now extend existence and uniqueness from the previous step to any finite time horizon.

For any $\tau\in [T_0,T]$, where $T_0 := T-c_0$ and $c_0>0$, let $\xi_\tau$ be such that $(\xi^x_\tau)_{x\in I}$ are e.p.i. and $\xi^x_\tau$ is $\mathcal{F}^x_\tau$-measurable for all $x\in I$. Assume that $c_0>0$ is small enough so that 
\begin{equation}
\label{eq:short-time-sys}
    \begin{aligned}
        &d\begin{bmatrix}
        \hat{X}_t^x
        \\
        \zeta^x_t
        \end{bmatrix}
        =
        \begin{bmatrix}
        \Gamma_{11}(x) + \eta^x_t\Gamma_{12}(x) & \Gamma_{12}(x)
        \\
        0 & \Gamma_{22}(x) - \eta^x_t\Gamma_{12}(x)
        \end{bmatrix}
        \begin{bmatrix}
        \hat{X}_t^x
        \\
        \zeta^x_t
        \end{bmatrix}
        dt
        \\
        &\hspace{2cm}+
        \begin{bmatrix}
        \Gamma_{Z,1}(x)
        \\
        \Gamma_{Z,2}(x) - \eta^x_t\Gamma_{Z,1}(x)  
        \end{bmatrix}\hat{Z}^x_t
        dt
        +
        \begin{bmatrix}
        1
        \\
        0
        \end{bmatrix}
        dB^x_t,& &x\in I,\ t\in [\tau, T],
        \\
        &
        \hat{X}^x_\tau = \xi^x_\tau,\quad
        \zeta^x_T =
        \Gamma_{T,2}(x)\hat{Z}^x_T, \quad \hat{Z}^x_t = \int_I w(x,y)\mathbb{E}[\hat{X}^y_t]\lambda(dy),& &x\in I,\ t\in[\tau,T].
    \end{aligned}
\end{equation}
has a unique solution as found in Step 2. Denote the solution $(\hat{X}^{0:x,\tau,\xi_\tau^\cdot}_t,\zeta^{0:x,\tau,\xi_\tau^\cdot}_t; t\in[\tau,T])$. 

Assume now that the forward-backward system \eqref{eq:the_first_eta-zeta} has a solution over the full time horizon: $(\hat{X}^x_t,\zeta^x_t; t\in[0,T])$. It is also a solution to \eqref{eq:short-time-sys} on the subinterval $[T_0,T]$ with $\hat{X}_{T_0}^x$ as initial condition at $T_0 = T-c_0$. By the unique solvability of \eqref{eq:short-time-sys},
$$
\mathbb{P}\boxtimes\lambda\Big((\omega,x) : (\hat{X}^x_t(\omega),\zeta^x_t) = (\hat{X}^{0:x,T_0,\hat{X}^\cdot_{T_0}}_t(\omega), \zeta^{0:x,T_0,\hat{X}^\cdot_{T_0}}_t),\ T_0 \leq t\leq T \Big) = 1.
$$ 
Now consider for some $\tau \in [0,T_0]$ the forward-backward system
\begin{equation}
\label{eq:short-time-sys2}
\begin{aligned}
        &d\begin{bmatrix}
        \hat{X}_t^x
        \\
        \zeta^x_t
        \end{bmatrix}
        =
        \begin{bmatrix}
        \Gamma_{11}(x) + \eta^x_t\Gamma_{12}(x) & \Gamma_{12}(x)
        \\
        0 & \Gamma_{22}(x) - \eta^x_t\Gamma_{12}(x)
        \end{bmatrix}
        \begin{bmatrix}
        \hat{X}_t^x
        \\
        \zeta^x_t
        \end{bmatrix}
        dt
        \\
        &\hspace{2cm}+
        \begin{bmatrix}
        \Gamma_{Z,1}(x)
        \\
        \Gamma_{Z,2}(x) - \eta^x_t\Gamma_{Z,1}(x)
        \end{bmatrix}\hat{Z}^x_t
        dt
        +
        \begin{bmatrix}
        1
        \\
        0
        \end{bmatrix}
        dB^x_t, & & x\in I,\ t\in [\tau, T_0],
        \\
        &
        \hat{X}^x_{\tau} = \xi^x_{\tau},\quad
        \zeta^x_{T_0} = \zeta^{0:x,T_0,\hat{X}^\cdot_{T_0}}_{T_0}, \quad  \hat{Z}^x_t = \int_I w(x,y)\mathbb{E}[\hat{X}^y_t]\lambda(dy),& &x\in I,\ t\in[\tau,T_0].
    \end{aligned}
\end{equation}
with $\xi^x_\tau$ satisfying the same assumptions as above but with the new $\tau \in [0, T_0]$ replacing the old $\tau \in [T_0,T]$. System \eqref{eq:short-time-sys2} differs from \eqref{eq:the_first_eta-zeta} in the terminal condition for the backward equation. In the next steps we deduce small-time existence and uniqueness of solutions of \eqref{eq:short-time-sys2}, we patch the solution with $(\hat{X}^{0:x,T_0,\hat{X}^\cdot_{T_0}}_t, \zeta^{0:x,T_0,\hat{X}^\cdot_{T_0}}_t; t\in[T_0,T])$, then we repeat and show that after a finite number of patching rounds we are left with the unique solution to \eqref{eq:the_first_eta-zeta}.
\\~\\
\textit{Step 4: Unique solvability of \eqref{eq:short-time-sys2} for short time horizons}\\
Let $\tau \in [0, T_0]$ and $E_{[\tau,T_0]} := C([\tau,T_0])$. Let $V_{[\tau,T_0]}$ and $V_{\boxtimes,[\tau,T_0]}$ denote $L^2(I; E_{[\tau,T_0]})$ and $L^2_\boxtimes(\Omega\times I; E_{[\tau,T_0]})$, respectively.
A fixed-point argument can be made to prove the existence of a $c_1>0$ such that if $T_0-\tau \leq c_1$, then there is a unique solution to \eqref{eq:short-time-sys2} in $V_{\boxtimes,T_0}\times V_{T_0}$. Denote the solution by $(\hat{X}^{1:x,\tau,\xi^\cdot_{\tau}}_t, \zeta^{1:x,\tau,\xi^\cdot_{\tau}}_t; t\in [\tau,T_0])$.

Let $T_1 := T_0 - c_1$. Most importantly (since we aim to use the induction approach) $c_1$ can take any value smaller than a constant $\bar C$ depending only on the time horizon $T$ and the coefficient function bound (from Assumption~\ref{assump:sec3.1}). The fixed-point calculations are tedious and omitted here. The main difficulty comes from that the terminal condition of the backward part of the equation depends on the solution of \eqref{eq:short-time-sys} initiated at the solution of \eqref{eq:short-time-sys2}. This can however be overcome by using the fixed-point argument for \eqref{eq:short-time-sys} from Step 2.
\\~\\
\textit{Step 5: Patching the solutions over $[T_1, T_0]$ and $[T_0, T]$}
\\
We now patch the two solutions. Let $\tau\in [T_1,T_0]$ and $\xi_\tau$ be the initial condition vector in \eqref{eq:short-time-sys2}. For any $s\in [\tau, T]$,
\begin{equation}
    (\hat{X}^x_s, \zeta^x_s) = 
    \begin{cases}
    (\hat{X}^{1:x,\tau,\xi^\cdot_{\tau}}_s, \zeta^{1:x,\tau,\xi^\cdot_{\tau}}_s), & s\in [\tau, T_0],
    \\
    (\hat{X}^{0:x,T_0,\hat{X}^{1:\cdot,\tau,\xi_\tau}_{T_0}}_s, \zeta^{0:x,T_0,\hat{X}^{1:\cdot,\tau,\xi_\tau}_{T_0}}_s), & s\in (T_0, T].
    \end{cases}
\end{equation}
Then, $\mathbb{P}\boxtimes\lambda$-a.s.,
\begin{align*}
    &\lim_{s\downarrow T_0} \hat{X}^x_s(\omega) = \lim_{s\downarrow T_0} \hat{X}^{0:x,T_0,\hat{X}^{1:\cdot,\tau,\xi_\tau}_{T_0}}_s(\omega) = \hat{X}^{1:x,\tau,\xi^\cdot_\tau}_{T_0}(\omega) = \hat{X}^x_{T_0}(\omega),
    \\
    &\lim_{s\downarrow T_0} \zeta^x_s = \lim_{s\downarrow T_0} \zeta^{0:x,T_0,\hat{X}^{1:\cdot,\tau,\xi_\tau}_{T_0}}_s = \zeta^{1:x,\tau,\xi^\cdot_\tau}_{T_0} = \zeta^x_{T_0},
\end{align*}
so $(\hat{X}^x_s, \zeta^x_s; \tau\leq s\leq T)$ is $\mathbb{P}\boxtimes\lambda$-a.s. continuous and the unique solution to the forward-backward system
\begin{equation*}
    \begin{aligned}
        &d\begin{bmatrix}
        \hat{X}_t^x
        \\
        \zeta^x_t
        \end{bmatrix}
        =
        \begin{bmatrix}
        \Gamma_{11}(x) + \eta^x_t\Gamma_{12}(x) & \Gamma_{12}(x)
        \\
        0 & \Gamma_{22}(x) - \eta^x_t\Gamma_{12}(x)
        \end{bmatrix}
        \begin{bmatrix}
        \hat{X}_t^x
        \\
        \zeta^x_t
        \end{bmatrix}
        dt
        \\
        &\hspace{2cm}+
        \begin{bmatrix}
        \Gamma_{Z,1}(x)
        \\
        \Gamma_{Z,2}(x) - \eta^x_t\Gamma_{Z,1}(x)
        \end{bmatrix}\hat{Z}^x_t
        dt
        +
        \begin{bmatrix}
        1
        \\
        0
        \end{bmatrix}
        dB^x_t, & & x\in I,\ t\in [\tau, T],
        \\
        &
        \hat{X}^x_{\tau} = \xi^x_{\tau},\quad
        \zeta^x_{T_0} = \zeta^{0:x,T_0,\hat{X}^\cdot_{T_0}}_{T_0}, \quad  \hat{Z}^x_t = \int_I w(x,y)\mathbb{E}[\hat{X}^y_t]\lambda(dy),& &x\in I,\ t\in[\tau,T].
    \end{aligned}
\end{equation*}
\textit{Step 6: The induction approach}\\
Consider the following forward-backward system: for some $\tau \in [0, T_1]$,
\begin{equation}
\label{eq:short-time-sys3}
    \begin{aligned}
        &d\begin{bmatrix}
        \hat{X}_t^x
        \\
        \zeta^x_t
        \end{bmatrix}
        =
        \begin{bmatrix}
        \Gamma_{11}(x) + \eta^x_t\Gamma_{12}(x) & \Gamma_{12}(x)
        \\
        0 & \Gamma_{22}(x) - \eta^x_t\Gamma_{12}(x)
        \end{bmatrix}
        \begin{bmatrix}
        \hat{X}_t^x
        \\
        \zeta^x_t
        \end{bmatrix}
        dt
        \\
        &\hspace{2cm}+
        \begin{bmatrix}
        \Gamma_{Z,1}(x)
        \\
        \Gamma_{Z,2}(x) - \eta^x_t\Gamma_{Z,1}(x) 
        \end{bmatrix}\hat{Z}^x_t
        dt
        +
        \begin{bmatrix}
        1
        \\
        0
        \end{bmatrix}
        dB^x_t, & & x\in I,\ t\in [\tau, T],
        \\
        &
        \hat{X}^x_{\tau} = \xi^x_{\tau},\quad
        \zeta^x_{T_1} =\zeta^{1:x,T_1,\hat{X}^\cdot_{T_1}}_{T_1}, \quad  \hat{Z}^x_t = \int_I w(x,y)\mathbb{E}[\hat{X}^y_t]\lambda(dy),& &x\in I,\ t\in[\tau,T].
    \end{aligned}
\end{equation}
Repeating the analysis that was done for the interval $[T_1,T_0]$ proves unique solvability for \eqref{eq:short-time-sys3} whenever $T_1 - \tau < c_1$, with $c_1$ being the same constant that was found above in Step 4. Hence, we have a unique solution to \eqref{eq:short-time-sys3} on the time interval $[T-(2c_1+c_0), T-(c_1+c_0)] =: [T_2,T_1]$. This solution can be patched with the solution on $[T_1,T]$ as was done in Step 5. After a finite number $N$ of iterations (an explicit lower bound on $N$ can be found, depending only on $T$ and $\bar C$, the constant from Step 4), the whole interval $[0,T]$ has been covered and the patching has yielded a unique solution to \eqref{eq:the_first_eta-zeta} over the interval $[0,T]$ for any finite $T>0$. 

\subsection{Proof of Proposition~\ref{prop:poc-conv-rate}}

The proof relies heavily on a bound for the following random variable: for a fixed $x^\infty\in I^\infty$ and $N\in \mathbb{N}$, we define 
\begin{equation*}
    \zeta^{x^\infty}_N : [0,T]\times I \ni (t,x) \mapsto \frac{1}{N}\sum_{h=1}^N  w(x,x_h)X^{x_h}_t - \int_I w(x,y)\mathbb{E}[X^y_t]\lambda(dy).
\end{equation*}
We know $\zeta^{x^\infty}_N$ is well-defined for all $x^\infty\in I^\infty$ since the graphon game state $X^x$ is defined for all $x\in I$.

\begin{lemma}
\label{lem:poc-zeta}
For all $x^\infty\in I^\infty$ and $N\in \mathbb{N}$ there exists a constant $C$, independent of $x^\infty$ and $N$, such that
\begin{align*}
    \max_{1\leq k\leq N}
    \Big(\mathbb{E}\Big[\sup_{t\in[0,T]}\big(|X^{k,N}_t - X^{x_k}_t|^2 \hspace{-3pt}
    + 
    \hspace{-3pt}
    |p^{k k,N}_t - p^{x_k}_t|^2\big) \Big]
    + \sup_{t\in[0,T]}\mathbb{E}\Big[|Z^{x_k,N}_t - Z^{x_k}_t|^2\Big]\Big)&
    \\
    \leq
        \sup_{(t,x)\in [0,T]\times I}C\Big(\mathbb{E}\Big[|\zeta^{x^\infty}_N(t,x)|^2\Big]
        + \frac{1}{N}\Big).&
\end{align*}
\end{lemma}

\begin{proof}
From standard estimates for linear SDEs and BSDEs we get
\begin{align*}
    &\mathbb{E}\Big[\sup_{t\in[0,T]}|X^{k,N}_t-X^{i_k}_t|^2 + \sup_{t\in[0,T]}|p^{k k,N}_t-p^{i_k}_t|^2
    \Big] + \sup_{t\in[0,T]}
    \mathbb{E}\left[|Z^{k,N}_t - Z^{i_k}_t|^2\right]
    \\
    &\leq
    C\Big(\sup_{(t,x)\in [0,T]\times I}\mathbb{E}\left[|\zeta^{x^\infty}_N(t,x)|^2\right]
    + \max_{1\leq \ell,h \leq N: \ell\neq h}
    \mathbb{E}\Big[\sup_{t\in[0,T]}|p^{\ell h,N}_t|^2
    \Big] + \frac{1}{N}\Big).
\end{align*}
Next, we will the estimate for the right hand side term containing off-diagonal adjoint state variables. Consider the following auxiliary BSDE system: for $k=1,\dots, N$, $p^{kk,N}_t = 0$ and $h = 1,\dots, N$, $h\neq k$,
\begin{equation}
\label{eq:aux-bsde}
        \widetilde p^{k h,N}_t = \int_t^T \Big(a(x_k)\widetilde p^{k h, N}_s + \frac{1}{N}\sum_{\ell=1}^N c(x_\ell)w(x_h,x_\ell)\widetilde p^{k \ell, N}_s\Big)ds - \sum_{\ell=1}^N\int_t^T\widetilde q^{k h \ell, N}_s dB^{x_\ell}_s.
\end{equation}
The difference $p^{kh,N}_t - \widetilde p^{kh,N}_t$, $1\leq h\neq k \leq N$, satisfies the BSDE
\begin{align*}
        &p^{k h, N}_t - \widetilde p^{k h, N}_t
        =
        \frac{w(x_k,x_h)}{N}
        \bigg(
        \bar\Gamma^*_T(x_k)\begin{bmatrix}
        X^{k,N}_T
        \\
        Z^{k, N}_T
        \end{bmatrix}
        +
        \int_t^T \bar\Gamma(x_k)\begin{bmatrix}
        X^{k,N}_s
        \\
        p^{kk,N}_s
        \end{bmatrix} 
        + \bar\Gamma_Z(x_k)Z^{k,N}_s ds \bigg)
        \\
        &\quad+ \int_t^T\Big(
        \frac{1}{N}\sum_{\ell=1}^N c(x_\ell)w(x_h,x_\ell)(p^{k\ell,N}_s-\widetilde p ^{k\ell,N}_s) + a(x_h) (p^{kh,N}_s-\widetilde p^{kh,N}_s) \Big)ds 
        \\
        &\quad 
        - \sum_{\ell=1}^N\int_t^T (q^{kh\ell,N}_s - \widetilde q^{kh\ell,N}_s)dB^{i_\ell}_s,\qquad t\in [0,T].
\end{align*}
Standard BSDE estimates (relying on the integrability of $X^{k,N}, Z^{k,N}$, and $p^{kk,N}$) yield
\begin{equation*}
    \mathbb{E}\Big[\sup_{t\in[0,T]}|p^{k h, N}_t - \widetilde p^{k h, N}_t|^2 + \int_0^T|q^{k h \ell, N}_s - \widetilde q^{k h \ell, N}_s|^2ds\Big] \leq \frac{C}{N},
\end{equation*}
with $C$ independent of $k,h,$ and $N$. The unique solution to \eqref{eq:aux-bsde} is $\widetilde p^{k h, N}_t = \widetilde q^{k h \ell}_t = 0$ for $t\in [0,T]$, $\ell = 1,\dots, N$, $1\leq k\neq h\leq N$. 
\end{proof}

Before attending to the first claim of the proposition, we establish a useful estimate. It follows from the Burkholder-Davis-Gundy inequality, Gronwall's lemma, and the uniform integrability of the initial conditions that
\begin{equation}
\label{lemma:uniform-bound-over-I-of-X}
    \sup_{x\in I} \mathbb{E}[\|X^x\|_E^2] \leq C(T,w),
\end{equation}
for some finite $C(T,w) > 0$ that depends only on $T$ and $w$. Adding and subtracting $\frac{1}{N}\sum_{k=1}^N w(x,x_k)\mathbb{E}[X^{x_k}_t]$ to $\zeta^{x^\infty}_N(t,x)$ results in
\begin{equation}
\label{eq-ineq}
    \begin{aligned}
        \mathbb{E}\left[|\zeta^{x^\infty}_N(t,x)|^2\right]
        &\leq
       C
    \mathbb{E}\Big[\Big|\frac{1}{N}\sum_{k=1}^N w(x,x_k)\left(X^{x_k}_t - \mathbb{E}[X^{x_k}_t]\right)\Big|^2\Big]
    \\
    &\qquad 
    +
    C
    \Big|\frac{1}{N}\sum_{k=1}^N w(x,x_k)\mathbb{E}[X^{x_k}_t] - 
    \int_I w(x,y)\mathbb{E}[X^y_t]\lambda(dy)\Big|^2.
    \end{aligned}
\end{equation}
By Lemma~\ref{lemma:extended-W},
$(w(x, x_k)(X^{x_k}_t - \mathbb{E}[X^{x_k}_t]))_{k=1}^\infty$ are mutually independent $\bar\lambda^\infty$-a.s. Thus $\bar\lambda^\infty\text{-a.s.}$
\begin{equation}
\label{eq:proof-prop-3-1}
    \mathbb{E}\Big[\Big|\frac{1}{N}\sum_{k=1}^N w(x,x_k)(X^{x_k}_t - \mathbb{E}[X^{x_k}_t])\Big|^2\Big]
    = \frac{1}{N^2}\sum_{k=1}^N
    \mathbb{E}\left[\left|w(x,x_k)(X^{x_k}_t - \mathbb{E}[X^{x_k}_t])\right|^2\right].
\end{equation}
The summands on the right hand side of \eqref{eq:proof-prop-3-1} can be bounded by a constant independent of $(t,x)$, using the estimate derived above. We get
\begin{equation*}
    \sup_{(t,x)\in[0,T]\times I}
    \mathbb{E}\Big[\Big|\frac{1}{N}\sum_{k=1}^N w(x,x_k)(X^{x_k}_t - \mathbb{E}[X^{x_k}_t])\Big|^2\Big]
    \leq \frac{C(T,w)}{N},\quad \bar\lambda^\infty\text{-a.s.}
\end{equation*}
for some constant $C(T,w)>0$ depending only on $T$ and $w$. We move on to the second term of the right hand side of \eqref{eq-ineq}. By the strong law of large numbers, it tends to zero almost surely as $N\rightarrow\infty$, proving the first claim of the proposition. To prove the second claim, we prove tightness of the term and then apply the Law of Iterated Logarithms.

Let $m(t,x) := \int_I w(x,y)\mathbb{E}[X^y_t]\lambda(dy)$
be the mean of $w(x,x_k)\mathbb{E}[X^{x_k}_t]$ when $x_k$ is $\lambda$-distributed. Let furthermore $\theta_k(x^\infty) : (t,x) \mapsto w(x,x_k)\mathbb{E}[X^{x_k}_t]-m(t,x)$. $\theta_k$ is a random variable on $(I^\infty,\mathcal{I}^\infty, \lambda^\infty)$ into $C([0,T]\times I)$.

\begin{lemma}
The collection $(\Theta_N)_{N\in \mathbb{N}}$, where $\Theta_N := \frac{1}{\sqrt{N}}\sum_{k=1}^N\theta_k$, is tight.
\end{lemma}

\begin{proof}
We will apply the tightness criterion provided by~\cite[Corollary 16.9]{MR1876169}, which stems from the Kolmogorov-Chentsov criterion. We first note that
$$
  \Theta_N(0,0) 
  = \frac{1}{\sqrt{N}}\sum_{k=1}^N\theta_k
  = \frac{1}{\sqrt{N}}\sum_{k=1}^N\left[ w(0,x_k)\mathbb{E}[X^{x_k}_0] - \int_I w(0,y)\mathbb{E}[X^y_0]\lambda(dy) \right],
$$
where the random variables are i.i.d. (and with mean $0$). So the sequence $(\Theta_N(0,0))_N$ is tight. 
Moreover, we prove that there exists a constant $C>0$ and there exists a positive integer $N_0\in\mathbb{N}$ such that for every $(t,x), (t',x') \in [0,T]\times I$, $N \ge N_0$,
\begin{equation*}    
\int_{I^\infty}\Big|
    \frac{1}{\sqrt{N}}\sum_{k=1}^N\theta_k(x^\infty)(t,x) - \theta_k(x^\infty)(t',x')\Big|^4
    \lambda^\infty(dx^\infty)
    \leq 
    C\left(|t-t'|^4 + |x-x'|^4\right).
\end{equation*}
As a consequence of~\cite[Corollary 16.9]{MR1876169}, we will obtain that the sequence of $(\Theta_N)_N$ is tight and the limiting processes are $\lambda^\infty$-a.s. locally H\"older continuous with exponent $1/2$. 

We now prove the claim. Let $(t,x), (t',x') \in [0,T]\times I$, $N\in\mathbb{N}$. Letting $\mathbf{T}_k := \theta_k(i^\infty)(t,x) - \theta_k(x^\infty)(t',x')$, we note that
\begin{align*}
    &\int_{I^N}\Big|
    \frac{1}{\sqrt{N}}\sum_{k=1}^N\theta_k(x^\infty)(t,x) - \theta_k(x^\infty)(t',x')\Big|^4
    \otimes_{k=1}^N\lambda(dx_k)
    \\
    &= \frac{1}{N^2} \int_{I^N} \left[\sum_{k_1,k_2,k_3,k_4=1}^N  \mathbf{T}_{k_1} \mathbf{T}_{k_2} \mathbf{T}_{k_3} \mathbf{T}_{k_4} \right] \otimes_{k=1}^N\lambda(dx_k)
    \\
    &= \underbrace{\frac{1}{N^2} \int_{I^N} \sum_{k=1}^N  \mathbf{T}_{k}^4 \otimes_{k=1}^N\lambda(dx_k)}_{\hbox{$\to 0 $ as  $N \to +\infty$ }}
    + \frac{1}{N^2} \sum_{\substack{k,k'=1 \\ k\neq k'}}^N  \int_{I^N} \left[\mathbf{T}_{k}^2 \mathbf{T}_{k'}^2 \right] \otimes_{k=1}^N\lambda(dx_k)
    \\
    &\qquad + \frac{1}{N^2}  \sum_{k_1=1}^N \underbrace{\int_{I^N}  \mathbf{T}_{k_1} \otimes_{k=1}^N\lambda(dx_k)}_{\hbox{ $= 0$ }} \sum_{\substack{k_2,k_3,k_4=1 \\ k_2,k_3,k_4 \neq k_1}}^N \int_{I^N} \left[\mathbf{T}_{k_2} \mathbf{T}_{k_3} \mathbf{T}_{k_4} \right] \otimes_{k=1}^N\lambda(dx_k).
\end{align*}
The first term can be made arbitrarily small by taking $N$ large enough. The third term is zero. To bound the second term from above, we observe that:
\begin{align*}
    \mathbf{T}_{k}^2 
    &=
    \big(\theta_k(x^\infty)(t,x) - \theta_k(x^\infty)(t',x')\big)^2
    \le C \Big( 
    |w(x,x_k) - w(x',x_k)|^2 \mathbb{E}[X^{x_k}_t]^2 
    \\
    & \quad + w(x',x_k)^2 \left| \mathbb{E}[X^{x_k}_t] - \mathbb{E}[X^{x_k}_{t'}] \right|^2
    + |m(t,x) - m(t',x')|^2 
    \Big)
    \\
    &\le C \Big(
    |x - x'|^2 
    +  \left| t - t' \right|^2
    + |m(t,x) - m(t',x')|^2 
    \Big).
\end{align*}
Furthermore, for the last term, we have by definition of $m$,
\begin{align*}
    &|m(t,x) - m(t',x')|^2 
    \\
    &\le \int_I\left( \left| w(x,y) - w(x',y) \right|^2 \mathbb{E}[X^y_t]^2 +  w(x',y)^2 \left|\mathbb{E}[X^y_t] - \mathbb{E}[X^y_{t'}]\right|^2 \right) \lambda(dy)
    \\
    &\le C \left(\left| x - x' \right|^2  +  \left|t - t'\right|^2 \right).
\end{align*}
Hence:
$$
    \frac{1}{N^2} \sum_{\substack{k,k'=1 \\ k\neq k'}}^N  \int_{I^N} \left[\mathbf{T}_{k}^2 \mathbf{T}_{k'}^2 \right] \otimes_{k=1}^N\lambda(di_k)
    \leq
    C \left(\left| x - x' \right|^4  +  \left|t - t'\right|^4 \right).
$$
\end{proof}

By Prokhorov's theorem (see \textit{e.g.}~\cite[Theorem 16.3]{MR1876169}), this yields relative compactness in distribution of $(\Theta_N)_N$. Moreover, the finite-dimensional distributions converge. Indeed, for any $n \in \mathbb{N}$ and any $r_1,\dots,r_n \in [0,T] \times I$, we have that the sequence $(\Theta_N(r_1), \dots, \Theta_N(r_n))_{N=1,2,\dots}$ converges in distribution by the standard central limit theorem. Hence, by~\cite[Lemma 16.2]{MR1876169}, $(\Theta_N)_N$ converges in distribution. By definition, this means that $\theta_1(x^\infty):(t,x) \mapsto w(x,x_1)\mathbb{E}[X^{x_1}_t]-m(t,x)$ satisfies CLT (see~\cite[Section 10.1]{ledoux2013probability}). Note that $\theta_k,k=2,3,\dots,$ are independent copies of $\theta_1$. Thus, $(\Theta_N)_N$ satisfies a Law of the Iterated Logarithm (see \textit{e.g.}~\citep[Theorem 10.12]{ledoux2013probability}). More precisely, we obtain the following.  

\begin{lemma}
\label{lemma:clt-gives-lil}
There exists a constant $C$ such that
\begin{equation*}
    \lambda^\infty\left(\underset{N\rightarrow\infty}{\lim\sup}\ \frac{\|\theta_1 + \dots + \theta_N\|_\infty}{\sqrt{2N\log\log N}} = C\right) = 1.
\end{equation*}
\end{lemma}

\begin{proof}
Using~\eqref{lemma:uniform-bound-over-I-of-X}, we have that
\begin{align*}
\|\theta_N\|^2_\infty 
&=
\sup_{(t,x)\in[0,T]\times I}|w(x,x_N)\mathbb{E}[X^{x_N}_t]-m(t,x)|^2
\\
&\leq
C(w)\left(\mathbb{E}[\|X^{x_N}\|_\infty^2]+\mathbb{E}^\boxtimes[\|X\|_\infty^2]\right)
\leq
C(w,T),
\end{align*}
where $C(w,T) > 0$ is a finite constant depending only on the graphon $w$ and $T$. Let $Lt := \max (1,\log t)$ for $t\geq 0$. Using the uniform bound derived above, we see that
\begin{equation*}
    \int_{I^\infty}\left(\frac{\|\theta_N\|^2_\infty}{LL\|\theta_N\|_\infty}\right)d\lambda^\infty(x^\infty) \leq 
    \int_{I^\infty}\|\theta_N\|^2_\infty d\lambda^\infty(x^\infty) < \infty.
\end{equation*}
The claim now follows from the Law of Iterated Logarithms in Banach spaces, see, \textit{e.g.}, \citep[Theorem 10.12]{ledoux2013probability}.
\end{proof}

In other words, Lemma~\hyperref[lemma:clt-gives-lil]{\ref{lemma:clt-gives-lil}} says that, for $\lambda^\infty$-a.e. $x^\infty\in I^\infty$: 
\begin{equation*}
    \forall \varepsilon > 0,\ \exists N_\varepsilon(x^\infty),\  \forall N\geq N_\varepsilon (x^\infty),\ \ \| \theta_1 + \dots + \theta_N\|_\infty \leq (C+\varepsilon)\sqrt{2N\log \log N},
\end{equation*}
and we have for $\lambda^\infty$-a.e. $x^\infty\in I^\infty$, for all $N \geq N_\varepsilon(x^\infty)$
\begin{equation*}
    \sup_{(t,x)\in[0,T]\times I} \Big|\frac{1}{N}\sum_{k=1}^N \left(w(x,x_k)\mathbb{E}[X^{x_k}_t] - m(t,x)\right)\Big|^2 
    \leq
    \frac{(C+\varepsilon)^2\log\log N}{N}.
\end{equation*}
The last statement also holds a.s. in $(I^\infty, \bar{\mathcal{I}}^\infty,\bar\lambda^\infty)$, see \citep[Section 6]{hammond2020monte}.

\subsection{Proof of Proposition~\ref{prop:eps-eq}}

We first prove the convergence claim. Recall that the two equilibria are linear functions of state, costate, and aggregate, so the propagation of chaos from Proposition~\ref{prop:poc-conv-rate} immediately yields
\begin{equation*}
    \max_{1\leq k\leq N}\mathbb{E}\Big[\int_0^T|\hat\alpha^{k,N}_t - \hat\alpha^{x_k}_t|^2dt\Big] \leq \varepsilon_N^2,\quad N\geq \underline{N},\ \bar\lambda^\infty\text{-a.s.}
\end{equation*}
Moving on to the approximation claim, let $\hat{\alpha}^{-k,N} := (\hat{\alpha}^{1,N},\dots, \hat{\alpha}^{k-1,N},\hat{\alpha}^{k+1,N},\dots, \hat{\alpha}^{N,N})$.
Since $(\hat{\alpha}^{1,N},\dots, \hat\alpha^{N,N})$ is a Nash equilibrium for the $N$-player game
\begin{align*}
    &J^{k,N}(\hat{\alpha}^{x_k}; \bar{\alpha}^{-k,N}) - J^{k,N}(\beta; \bar{\alpha}^{-k,N}) 
    \\
    &\leq |J^{k,N}(\hat\alpha^{x_k}; \bar\alpha^{-k,N}) - J^{k,N}(\hat\alpha^{k,N}; \hat\alpha^{-k,N})| + |J^{k,N}(\beta; \bar\alpha^{-k,N}) - J^{k,N}(\beta; \hat\alpha^{-k,N})|.
\end{align*}
Let $\beta$ be an admissible control for player $1$ in the $N$-player game. Let $(\bar X^{k,N,\beta})_{k=1}^N$ be the player states when player $1$ is using $\beta$ and the others are using the graphon game equilibrium control
    \begin{align*}
        &d\bar X^{1,N,\beta}_t = \Big(a(x_k)\bar X^{1,N,\beta}_t + b(x_k)\beta_t + c(x_k)\frac{1}{N}\sum_{h=1}^Nw(x_k,x_h)\bar X^{h,N,\beta}_t\Big)dt + dB^{x_1}_t,
        \\
        &d\bar X^{k,N,\beta}_t = \Big(a(x_k)\bar X^{k,N,\beta}_t + b(x_k)\hat \alpha^{x_k}_t + c(x_k)\frac{1}{N}\sum_{h=1}^Nw(x_k,x_h)\bar X^{h,N,\beta}_t\Big)dt + dB^{x_k}_t,
        \\ 
        &\hat X^{1,N,\beta}_0 = \xi^{x_1},\quad \bar X^{k,N,\beta}_0 = \xi^{x_k},\quad k\geq 2.
    \end{align*}
Let $(\hat X^{k,N,\beta})_{k=1}^N$ be the player states when player $1$ is using $\beta$ and the others are using the $N$-player game equilibrium control
\begin{align*}
        &d\hat X^{1,N,\beta}_t = \Big(a(x_k)\hat X^{1,N,\beta}_t + b(x_k)\beta_t + c(x_k)\frac{1}{N}\sum_{h=1}^Nw(x_k,x_h)\hat X^{h,N,\beta}_t\Big)dt + dB^{x_1}_t,
        \\
        &d\hat X^{k,N,\beta}_t = \Big(a(x_k)\hat X^{k,N,\beta}_t + b(x_k)\hat\alpha^{k,N}_t + c(x_k)\frac{1}{N}\sum_{h=1}^Nw(x_k,x_h)\hat X^{h,N,\beta}_t\Big)dt + dB^{x_k}_t, 
        \\  
        &\hat X^{1,N,\beta}_0 = \xi^{x_1},\quad \hat X^{k,N,\beta}_0 = \xi^{x_k},\quad k\geq 2.
\end{align*}
We have for any admissible $\beta$ and $\bar\lambda^\infty$-a.e. $x^\infty\in I^\infty$ that
\begin{align*}
    &|J^{1,N}(\beta; \bar\alpha^{-1,N}) - J^{1,N}(\beta; \hat\alpha^{-1,N})|
    \\
    &\leq
    \mathbb{E}\Big[\int_0^T\big|f^{x_1}(\bar X^{1,N,\beta}_t,\beta_t, \bar Z^{1,N,\beta}_t) - f^{x_1}(\hat X^{1,N,\beta}_t, \beta_t, \hat Z^{1,N,\beta}_t)\big|dt\Big] 
    \\
    &\qquad + \mathbb{E}\Big[\big|h^{x_1}(\bar X^{1,N,\beta}_T, \bar Z^{1,N,\beta}_T) - h^{x_1}(\hat X^{1,N,\beta}_T, \hat Z^{1,N,\beta}_T)\big|\Big]
    \\
    &\leq C \max_{1\leq k\leq N}\mathbb{E}\Big[\int_0^T\big|\hat\alpha^{x_k}_t - \hat\alpha^{k,N}_t\big|^2dt\Big]^{1/2} \leq \varepsilon_N,\quad N\geq \underline{N}(x^\infty).
\end{align*}
In the same way we can perturb the other players' actions and that concludes the proof.

\subsection{Conditions for well-posedness of the time-varying Riccati equation \eqref{eq:riccati_for_pi}}

Let 
\begin{equation*}
    \mathbb{A}_t := \Gamma_{Z,2} - \Gamma_{Z,1}\eta_t,
    \ 
    \mathbb{B}^k_t := -\Gamma_{12}\eta_t -\frac{1}{2}(\Gamma_{11} - \Gamma_{22} + \Gamma_{Z,1}\lambda_k), 
    \ 
    \mathbb{C}^k :=  -\Gamma_{12}\lambda_k, 
    \  
    t\in [0,T].
\end{equation*}
Set $\mathbb{D}^k := (\mathbb{B}^k_t)^2 + \mathbb{C}^k \mathbb{A}_t - \dot{\mathbb{B}}^k_t$. Note that $\mathbb{D}^k$ does not depend on $t$. Indeed, using the ODE satisfied by $\eta$ (\textit{i.e.}, ~\eqref{eq:the_first_eta-zeta} but with constant coefficients) and the identity $\Gamma_{11} = -\Gamma_{22}$, we have
\begin{align*}
    \mathbb{D}^k
    &= 
    \frac{1}{4}(2\Gamma_{11} + \Gamma_{Z,1}\lambda_k)^2 -\Gamma_{12}\lambda_k\Gamma_{Z,2} + \Gamma_{12} \Gamma_{21}.
\end{align*} 
Last, we let:
$$
    \mathbb{F}^k := \frac{\mathbb{B}^k_T + \mathbb{C}^k \Gamma_{T,2} + \sqrt{\mathbb{D}^k}}{\mathbb{B}^k_T + \mathbb{C}^k \Gamma_{T,2} - \sqrt{\mathbb{D}^k}}.
$$

\begin{proposition}
    Assume $\mathbb{D}^k \ge 0$ and,  for all $t \in [0,T]$,  $\mathbb{F}^k e^{\sqrt{\mathbb{D}^k}t} - e^{-\sqrt{\mathbb{D}^k}t} \neq 0$. 
    Then the Riccati equation~\eqref{eq:riccati_for_pi}  has a unique solution.
\end{proposition}

\begin{remark}
Note that we can rewrite $\mathbb{D}^k$ as:
\begin{align*}
    \mathbb{D}^k
    &= 
    \bigg(\Gamma_{11} -\frac{\Gamma_{12} \Gamma_{Z,2}}{\Gamma_{Z,1}} + \frac{1}{2}\Gamma_{Z,1} \lambda_k \bigg)^2 - \bigg(\Gamma_{11} -\frac{\Gamma_{12} \Gamma_{Z,2}}{\Gamma_{Z,1}}\bigg)^2 + (\Gamma_{11})^2 + \Gamma_{12}\Gamma_{21}.
\end{align*}
So, to ensure $\mathbb{D}^k \ge 0$ independently of the value of $\lambda_k$, a sufficient condition is:
\begin{equation*}
     (\Gamma_{11})^2 
     \ge  \bigg(\Gamma_{11} -\frac{\Gamma_{12}\Gamma_{Z,2}}{\Gamma_{Z,1}}\bigg)^2 - \Gamma_{12} \Gamma_{21}.
\end{equation*}
\end{remark}

\begin{proof}
\textbf{Existence: } We first consider the ODE with the mixed initial condition 
\begin{equation}
\label{eq:riccati_for_nu}
    \ddot{\nu}^k_t 
    - \mathbb{D}^k \nu^k_t = 0,\quad t\in [0,T],\qquad \nu^k_0 (\mathbb{B}^k_T + \mathbb{C}^k \Gamma_{T,2}) - \dot{\nu}^k_0 = 0.
\end{equation}
A solution is given by $\nu^k_t = \mathbb{F}^k e^{\sqrt{\mathbb{D}^k}t} -  e^{-\sqrt{\mathbb{D}^k}t}$. Let $\theta^k_t = \nu^k_{T-t}$, for $t \in [0,T]$. It solves the ODE  with the mixed terminal condition:
\begin{align*}
    \ddot{\theta}^k_t 
    = 
    \mathbb{D}^k \theta^k_t,\quad t\in [0,T],\qquad \theta^k_T (\mathbb{B}^k_T + \mathbb{C}^k \Gamma_{T,2}) + \dot{\theta}^k_T = 0.
\end{align*}
To conclude, let
\begin{equation*}
    \pi^k_t = -\frac{1}{\mathbb{C}^k}\bigg( \frac{\dot{\theta}^k_t}{\theta^k_t} + \mathbb{B}^k_t \bigg),
    \qquad
    \dot{\pi}^k_t = -\frac{1}{\mathbb{C}^k}\bigg( \frac{\ddot{\theta}^k_t}{\theta^k_t} - \bigg(\frac{\dot{\theta}^k_t}{\theta^k_t}\bigg)^2 + \dot{\mathbb{B}}^k_t \bigg), \qquad t \in [0,T].
\end{equation*}
Then it can be checked that $\pi^k_t$ solves
\begin{equation*}
    \dot{\pi}^k_t 
    = 
     - \mathbb{A}_t + 2 \mathbb{B}^k_t \pi^k_t + \mathbb{C}^k (\pi^k_t)^2,
     \quad
    \pi^k_T 
    =
    [C_h]_{12},
\end{equation*}
which is equivalent to~\eqref{eq:riccati_for_pi}.

\textbf{Uniqueness: } Let us consider $\pi^k$ and $\tilde\pi^k$ solving~\eqref{eq:riccati_for_pi}. Reverting the above change of variables yields solutions $\nu^k$ and $\tilde\nu^k$ to~\eqref{eq:riccati_for_nu} such that $\tilde\nu^k_t \neq 0$ and $\nu^k_t \neq 0$ for all $t \in [0,T]$. For any such solutions, there exist constants $C_1, C_2, \tilde C_1, \tilde C_2$ such that
\begin{equation*}
    \nu^k_t = C_1 e^{\sqrt{\mathbb{D}^k}t} + C_2 e^{-\sqrt{\mathbb{D}^k}t},
    \qquad
    \tilde\nu^k_t = \tilde C_1 e^{\sqrt{\mathbb{D}^k}t} + \tilde C_2 e^{-\sqrt{\mathbb{D}^k}t}.
\end{equation*}    
Due to~\eqref{eq:riccati_for_nu}, we have:
\begin{align*}
    &C_1 (\mathbb{B}^k_T + \mathbb{C}^k \Gamma_{T,2} - \sqrt{\mathbb{D}^k}) + C_2 (\mathbb{B}^k_T + \mathbb{C}^k \Gamma_{T,2} + \sqrt{\mathbb{D}^k}) = 0,
    \\
    &\tilde C_1 (\mathbb{B}^k_T + \mathbb{C}^k \Gamma_{T,2} - \sqrt{\mathbb{D}^k}) + \tilde C_2 (\mathbb{B}^k_T + \mathbb{C}^k \Gamma_{T,2} + \sqrt{\mathbb{D}^k}) = 0.
\end{align*}
As a consequence, we necessarily have $C_1/C_2 = \tilde C_1 / \tilde C_2 = - \mathbb{F}^k$.  We deduce that $\pi^k_t = \tilde\pi^k_t$ for all $t \in [0,T]$.

\end{proof}

\end{appendix}

\bibliographystyle{chicago}

\end{document}